\documentclass[11pt]{amsart}%
\usepackage{amsmath}%
\usepackage{amssymb}%
\usepackage{amscd}%
\usepackage{color}%
\input xy \xyoption{all}
\def\wh#1{\widehat{#1}}
\def\wt#1{\widetilde{#1}}
\theoremstyle{plain}
    \newtheorem{theorem}{Theorem}[section]
    
    \newtheorem{proposition}[theorem]{Proposition}
    \newtheorem{lemma}[theorem]{Lemma}
    \newtheorem{corollary}[theorem]{Corollary}
\theoremstyle{definition}
    \newtheorem{definition}[theorem]{Definition}

    \newtheorem{remark}[theorem]{Remark}
    \newtheorem{conjecture}[theorem]{Conjecture}

\def\Alphabet{A,B,C,D,E,F,G,H,I,J,K,L,M,N,O,P,Q,R,S,T,U,V,W,X,Y,Z}
\def\alphabet{a,b,c,d,e,f,g,h,i,j,k,l,m,n,o,p,q,r,s,t,u,v,w,x,y,z}
\def\endpiece{xxx}
\def\makeAlphabet[#1]{\expandafter\makeA#1,xxx,}
\def\makealphabet[#1]{\expandafter\makea#1,xxx,}
\def\makeA#1,{\def\temp{#1}\ifx\temp\endpiece\else%
\mkbb{#1}\mkfrak{#1}\mkbf{#1}\mkcal{#1}\expandafter\makeA\fi}%
\def\makea#1,{\def\temp{#1}\ifx\temp\endpiece\else\mkfrak{#1}\mkbf{#1}\expandafter\makea\fi}%
\def\mkbb#1{\expandafter\def\csname bb#1\endcsname{\mathbb{#1}}}
\def\mkfrak#1{\expandafter\def\csname fr#1\endcsname{\mathfrak{#1}}}
\def\mkbf#1{\expandafter\def\csname b#1\endcsname{\mathbf{#1}}}
\def\mkcal#1{\expandafter\def\csname c#1\endcsname{\mathcal{#1}}}
\def\makeop[#1]{\xmakeop#1,xxx,}
\def\mkop#1{\expandafter\def\csname #1\endcsname{{\mathrm{#1}}}} %
\def\xmakeop#1,{\def\temp{#1}\ifx\temp\endpiece\else\mkop{#1}\expandafter\xmakeop\fi}%
\makeAlphabet[\Alphabet]
\makealphabet[\alphabet]
\makeop[Alt,End,Ext,Hom,Sym,Tor,Aut,Fil,Map]
\makeop[CH,Pic,div,Det,Div]
\makeop[Ker,Im,Coker,Coim,id,pr,isom]
\makeop[Spec,Spf,Spm]
\makeop[Re,Im]%
\makeop[dR,Nis,crys,cris,rig,syn,tor,Cont,cont,et]
\makeop[Gal,Res,ab]
\makeop[pol]
\makeop[Var,an,Isoc,univ]%
\makeop[Gl,Sl]
\makeop[Lie]
\makeop[tay,tr,rk,ind,N,Rad,tors]
\makeop[Sat]
\makeop[Cor]
\makeop[Frob,Fr] 

\def\isom{\cong}
\def\verk{\circ}

\def\prolim{\varprojlim}

\begin{document}

\title[Equivariant main conjecture]{On the equivariant and the non-equivariant main conjecture for imaginary quadratic fields}   
\author{Jennifer Johnson-Leung and Guido Kings}
\begin{abstract}
	In this paper we first prove the main conjecture for imaginary quadratic fields for all
	prime numbers $p$, improving slightly earlier results by Rubin. From this we deduce the equivariant
	main conjecture in the case that a certain $\mu$-invariant vanishes. For prime numbers 
	$p\nmid 6$ which split in $K$, we can prove the equivariant main conjecture using
	a theorem by Gillard.
\end{abstract}
\date{\today} 
\maketitle
\setcounter{tocdepth}{1}
\tableofcontents

\section*{Introduction}
The Iwasawa main conjecture has been an important 
tool to study the arithmetic of special values of $L$-functions of Hecke characters
of imaginary quadratic fields (\cite{Rubin}, \cite{Kings}, \cite{Ts}, \cite{Bley}, \cite{JLthesis}).
To obtain the finest possible invariants it is important to know the
main conjecture for all prime numbers $p$ and also to have an equivariant version
at disposal.

In this paper we address these questions and treat the main conjecture for imaginary quadratic fields $K$
in the equivariant and the non-equivariant setting. 
Our results are twofold:

As a first theorem (see \ref{lambda-mc}), we prove the main conjecture first proven by Rubin
\cite{Rubin}, \cite{RubinMoreMC} for all prime numbers $p$. This improves the
results by Rubin, who had to impose the condition that $p$ does not divide the order of the abelian
field defined by $\chi$. 

The second result of our paper treats the equivariant main conjecture. We reduce this
conjecture to the vanishing of a certain $\mu$-invariant (see \ref{mu=0} for the precise condition).
A result of Gillard \cite{Gillard} implies that the equivariant main conjecture
is a theorem for prime numbers $p\nmid 6$, which split in $K$.
 
It was Rubin's idea to prove the main conjecture with the techniques of Euler systems
invented by Kolyvagin (building also on ideas of Thaine). Later, he (and also Kato and Perrin-Riou independently) developed
the machinery in an abstract and conceptual way, which made it a very flexible and
general tool.

Our approach to the main conjecture follows the scheme of proof developed by
the second author with A. Huber in \cite{HK1}. Instead of decomposing the
classical Iwasawa modules under character-wise projectors (some of which turn out to not be integral and lead to restrictions on $p$), we use Galois cohomology with coefficients
in the Galois representations defined by the character $\chi$. Using this
we reduce the main conjecture to the Tamagawa number conjecture for number fields
at $s=0$.  We then exploit the isogeny invariance of the TNC to reduce to the analytic class number formula.
This approach was inspired by the Tamagawa number conjecture and in particular by the work of Kato.

To treat the equivariant main conjecture, Burns and Greither had the happy idea that
the vanishing of certain $\mu$-invariants had the consequence that the decisive Iwasawa modules
vanish when localized at so called singular prime ideals (see \ref{proofforsingular}).
We essentially adopt this strategy but with a conceptual change first explained
by Witte \cite{Witte}: we deduce the equivariant main conjecture from the characterwise
one using the fact that the vanishing of the $\mu$-invariant implies the vanishing
of the localized $H^2$, which is essentially the inverse limit of the class groups. 

For the experts we like to point out one seemingly new technical feature in the proof.
Kato had the idea that one should use the functor $\Det$ of Knudsen and Mumford \cite{knumum}
instead of the more traditional characteristic ideal. We not only follow his suggestion, but
we use also the functor $\Div$ in a systematic way.
This allows us to deal in an elegant way with the
reduction of the main conjecture to the Tamagawa number conjecture.

The paper is organized as follows: after some notational preliminaries in the first section,
we review the Tamagawa number conjecture for number fields at $s=0$. The next section
recalls the Euler system of elliptic units following the exposition by Kato in \cite{Kato}.
The fourth section introduces the basic Iwasawa modules and studies some of their properties.
The technical part here is simpler than in the corresponding case of 
the main conjecture for $\bbQ$, as we work here with a $\bbZ_p^2$-extension of
$K$. This implies that some cohomological Iwasawa modules are automatically pseudo-null,
which is not true in the case for $\bbQ$. The fifth section formulates
the equivariant (here called $\Omega$-main-conjecture) and the non-equivariant 
Iwasawa main conjecture (here called $\Lambda$-main-conjecture). The last two sections
contain the proofs of these main conjectures.

The first named author would like to thank Universit\"at Regensburg for its hospitality during the initial phase of this project.  The second named author would like to thank M. Witte for very useful discussions
concerning \cite{Witte}.
It is a great pleasure for both authors to thank the referee for a very careful reading of the manuscript and
for pointing out several inaccuracies in an earlier version of the text.

\section{Preliminaries}

\subsection{General notations}\label{notations}

In this paper $K$ always denotes an imaginary quadratic field with a fixed embedding 
$K\to \bbC$ and we fix an algebraic closure $\bar{K}\subset \bbC$. By $\cO_K$ we denote the ring
of integers. For each ideal $\frf\subset \cO_K$ we consider the ray class field $K(\frf)$ of
modulus $\frf$ and we denote by
$$
	G(\frf):=\Gal(K(\frf)/K)
$$
its Galois group over $K$. 
Consider for an ideal $\frf\subset \cO_K$  characters
$$
	\eta:G(\frf)\to \bbC^*.
$$
We will use in this paper the \emph{inverse} of the usual 
Artin reciprocity map. That is, if $\frq$ is a prime ideal, which does not divide the conductor of  $K(\frf)$,
then we associate to it $\Frob_\frq^{-1}$ the geometric Frobenius. Thus we get
$$
	\eta(\frq)=\eta(\Frob_\frq^{-1}).
$$
To have compatibility with the sources \cite{Kato} and \cite{deShalit}, we let 
$\sigma_\fra$ be the usual Artin symbol, so that 
$$
	\eta(\fra)=\eta^{-1}(\sigma_\fra).
$$
The conductor of $\eta$ will be denoted by $\frf_\eta$ and we let 
$$
	\wh{G}(\frf):=\{\eta:G(\frf)\to \bbC^*\}
$$
be the dual group of $G(\frf)$.  We denote by $E$ a number field, which contains the values of $\eta$.
We denote by $\cO:=\cO_E$ the ring of integers in $E$ and we introduce the following conventions:
$$
	E_\infty:=E\otimes_\bbQ\bbR\mbox{ and } E_p:=E\otimes_\bbQ\bbQ_p.
$$
In a similar way we let $\cO_p:=\cO\otimes_\bbZ\bbZ_p$. Note that this is a product of discrete valuation
rings.

For each character $\eta:G(\frf)\to E^*$ and each embedding $\sigma:E\to \bbC$ we define the $E\otimes_\bbQ\bbC\simeq\Hom(E,\bbC)$-valued $L$-function of $\eta$ for $\Re(s)> 1$
to be 
$$
	L(\eta,s):=(\ldots,L(\sigma\verk\eta,s),\ldots),
$$
where
$$
L(\sigma\verk\eta,s):=\prod_{0\neq\frp\subset\cO_K}\frac{1}{1-\frac{\sigma\eta(\frp)}{\N(\frp)^s}},
$$
and the product is taken over all non-trivial prime ideals of $\cO_K$.
For each ideal $\frn\subset \cO_K$ we define
$$
	L_\frn(\sigma\verk\eta,s):=\prod_{\frp\nmid \frn}\frac{1}{1-\frac{\sigma\eta(\frp)}{\N(\frp)^s}}.
$$
These $L$-functions have a meromorphic continuation to $\bbC$ and satisfy a functional equation.
If $\eta\neq 1$ is non-trivial, the functions $L(\sigma\verk\eta,s)$ 
have a zero of order $1$ at $s=0$. We write 
$$
	L^*(\eta,0)
$$
for the leading term in the Laurent series at $0$ of $L(\eta,s)$ as an $E\otimes\bbC$-valued function.

\subsection{The motive of a number field}\label{motivic}

For each Galois extension $K\subset F\subset \bar{K}$ with Galois group $G:=\Gal(F/K)$, we denote by $h^0(F)$ its motive over $K$ and
$$
	M(F):=h^0(F)_E
$$ 
the motive with coefficients in $E$. Here we assume that $E$ contains all the values of the characters
in  $\wh{G}$.
For each group $G$ and a commutative ring $R$, we let 
$$
	R[G]
$$
be the group ring of $G$ with coefficients in $R$. Now suppose that $G$ is abelian.
For a character $\eta:G\to E^*$
we let 
$$
	p_{\eta^{-1}}:=\frac{1}{\#G}\sum_{\sigma\in G}\eta(\sigma)\sigma\in E[G]
$$	
be the projector onto the $\eta^{-1}$-eigenspace.
The projectors $p_{\eta^{-1}}$
decompose $M(F)$ into a direct sum
$$
	M(F)=\bigoplus_{\eta\in \wh{G}}M(\eta),
$$
where 
$$
	M(\eta):=p_{\eta^{-1}}M(F).
$$
Note that if $\eta$ factors through a subgroup $G'$, then also $p_{\eta^{-1}}$ factors
through $E[G']$. If $G'=\Gal(F'/K)$, then
\begin{equation}\label{motivdec}
M(\eta):=p_{\eta^{-1}}M(F)=p_{\eta^{-1}}M(F').
\end{equation}
This means that $M(\eta)$ is independent of the choice of the
group on which the character $\eta$ is considered.
The $L$-function of the motive $M(F)$ is the Dedekind zeta function of $F$,
$$
L(M(F),s)=\zeta_{F}(s)
$$
considered as $E\otimes_\bbQ \bbC$-valued function. Similarly, 
$$
	L(M(\eta),s)=L(\eta,s)
$$
for each character $\eta:G\to E^*$.
We consider several realizations of the motives $M(F)$ and the dual motive $M(F)^\vee(1)$ with a Tate twist. 
In this case, since the dimension of the variety is 0, $M(F)^\vee=M(F)$.
Note that the dual motive of $M(\eta)$ is
$$
	M(\eta)^\vee\simeq M(\eta^{-1}).
$$
The \emph{Betti realization} is the $E$-vector space
$$
M(F)_{B}:=H^0_B(\Spec F(\bbC),E)\simeq\bigoplus_{\tau:F\to_K \bbC}E\simeq E[G].
$$ 
Here the sum is over all embeddings of $F$ into $\bbC$, which agree with the
fixed embedding of $K$. Also we have used the fixed embedding of $F\subset \bar{K}$ into $\bbC$
in the last isomorphism.

The \emph{deRham realization} 
$$
M(F)_{\dR}:=H^0_{\dR}(F/K)\otimes_\bbQ E\simeq F\otimes_\bbQ E
$$
is a filtered $K\otimes_\bbQ E$-module, and in this case, $\Fil^0(M(F)_{\dR})=M(F)_{\dR}$.  
The \emph{\'etale realization} for any prime number $p$, 
$$
M(F)_p:=H^0_{\et}(\Spec F\times_K\bar{K},E_p)\simeq \bigoplus_{\tau:F\to_K \bar{K}}E_p\simeq E_p[G]
$$
is an $E_p$-representation of $\Gal(\bar{K}/K)$. 

The \emph{motivic cohomology} groups are defined in terms of $K$-theory and we have
 $H^0_f(M(F))=E$ and $H^1_f(M(F))=0$ while $H^0_f(M(F)^\vee(1))=0$ and 
$$
	H^1_f(M(F)^\vee(1))= K_1(\cO_{F})\otimes_\bbZ E\simeq\cO^*_{F}\otimes_\bbZ E. 
$$
The realizations of the motives $M(\eta)$ are defined by applying the projector
$p_{\eta^{-1}}$ to the realizations of $M(F)$. In particular, we have
$$
	H^1_f(M(\eta)^\vee(1))=p_{\eta}(\cO^*_{F}\otimes_\bbZ E).
$$
\begin{definition}\label{latticedefn}
Using the identification $E[G]\simeq M(F)_B$ we define the 
\emph{canonical lattice} to be $\cO[G]\subset M(F)_B$. Similarly, we consider
$\cO_p[G]\subset M(F)_p$. This induces a \emph{canonical lattice}
$$
	\cO(\eta):=p_{\eta^{-1}}(\cO[G])\subset M(\eta)_B
$$
with \emph{canonical generator} $t_B(\eta):=p_{\eta^{-1}}(1)$ and 
Galois stable lattices
$$
	\cO_p(\eta):=p_{\eta^{-1}}(\cO_p[G])\subset M(\eta)_p
$$
with \emph{canonical generator} $t_p(\eta):=p_{\eta^{-1}}(1)$.
We also define
$$
	\cO(\eta)^\vee:=\Hom_{\cO}(\cO(\eta),\cO)
$$
and
$$
	\cO_p(\eta)^\vee:=\Hom_{\cO_p}(\cO_p(\eta),\cO_p)
$$
and denote by $t_B(\eta)^\vee$ and $t_p(\eta)^\vee$ the
dual bases.
\end{definition}
Note that all these notions depend on $\eta$ and in the following
way on $G$: Suppose that $\eta$ factors into $\eta=\wt\eta\verk \pi$, where
$\pi:G\to G'$ is a group homomorphism. Then $\pi$ induces a map $\pi:E[G]\to E[G']$ and one has
$$
\pi(p_{\eta^{-1}}(\cO[G]))=p_{\wt\eta^{-1}}(\cO[G']).
$$
In this sense the lattice $\cO(\eta)$ is independent of the group $G$, where $\eta$ is defined.

Note also that the action of $\Gal(\bar{K}/K)$ on $M(F)_p$ factors through $G$ but this
action is contragredient to the canonical action of $G$ on $M(F)_p$. This is 
the reason why $M(\eta)_p=p_{\eta^{-1}}M(F)_p$ has Galois action through $\eta$. 

\subsection{The functors $\Det$ and $\Div$}\label{determinants}
For any ring $R$ a \emph{perfect complex} is a complex of (left) $R$-modules,
which is quasi-isomorphic to a bounded complex of finitely generated projective
$R$-modules.

We will use the graded determinant functor $\Det$ and the divisor functor
$\Div$ of Knudsen and Mumford \cite{knumum}.  Let $R$ now be a commutative ring, and  
$$
P^{\cdot}: \cdots\rightarrow P^{i-1}\rightarrow P^i\rightarrow P^{i+1}\rightarrow\cdot\cdot\cdot
$$ a perfect complex of projective $R$-modules. 
One defines $\Det_RP^i:=\bigwedge_{R}^{\mbox{rk}_{R}{P^i}}P^i$ as a graded invertible $R$-module of
(locally constant) degree $\mbox{rk}P^i$.
The determinant of the complex $P^\cdot$ is then the graded invertible $R$-module 
$$
\Det_RP^{\cdot}:=\bigotimes_{i\in\bbZ}\Det_R^{(-1)^i}P^i.
$$
Notice that the determinant depends only on the quasi-isomorphism class of $P^\cdot$.  Moreover, 
if the cohomology groups $H^i(P^\cdot)$ are all perfect, one has
$$
\Det_RP^\cdot=\bigotimes_{i\in\bbZ}\Det_R^{(-1)^i}H^i(P^\cdot).
$$
This functor is closely related to the characteristic ideal.  
If $P$ is a torsion $R$-module, $R$ a regular noetherian integral domain and $Q(R)$ the total quotient ring of $R$, then $\mbox{char}(P)=\Det_R^{-1}P$. Here we identify 
$$
	\Det_R^{-1}P\subset (\Det_R^{-1}P)\otimes_RQ(R)=\Det_{Q(R)}^{-1}0=Q(R).
$$
Assume now that $R$ is noetherian and let
$$
\lambda:\cF^{\cdot}\to \cG^{\cdot}
$$ 
be a map of perfect complexes on  $X:=\Spec R$ in
the derived category. Let $U(\lambda)$ be the open set of $x\in X$ such that
$\lambda$ is an isomorphism in a neighbourhood of $x$. The map $\lambda$ is called good if
$U(\lambda)$ contains all points of depth $0$. Knudsen and Mumford define
for good $\lambda$ a Cartier divisor $\Div(\lambda)$ on $X$, which has the property that the canonical map
on $U(\lambda)$
$$
	\Det(\lambda):\Det(\cF^{\cdot})\mid_{U(\lambda)}\simeq \Det(\cG^{\cdot})\mid_{U(\lambda)}
$$
extends to an isomorphism on the whole of $X$
\begin{equation}\label{divequation}
	\Det(\lambda):\Det(\cF^{\cdot})(\Div(\lambda))\simeq \Det(\cG^{\cdot}).
\end{equation}
In particular, one has an isomorphism $\cO_X(\Div(\lambda))\simeq \Det(\cG^{\cdot})\otimes \Det^{-1}(\cF^{\cdot})$.
One defines
$$
	\Div(\cF^{\cdot}):=\Div(0\to \cF^{\cdot}),
$$
if $0\to \cF^{\cdot}$ is good.
The functor $\Div$ has among other the following properties (see \cite{knumum} Theorem 3):
If 
$$
	0\to \cF^{\cdot}\xrightarrow{\lambda}\cG^{\cdot}\xrightarrow{\mu}\cH^{\cdot}\to 0
$$
is a short exact sequence of perfect complexes such that $\lambda$ is good, then
$0\to \cH^{\cdot}$ is good and $\Div(\lambda)=\Div(\cH^{\cdot})$.

If $\lambda:\cF^{\cdot}\to \cG^{\cdot}$ and $\mu:\cH^{\cdot}\to \cI^{\cdot}$ are good,
then $\Div(\lambda\oplus \mu)=\Div(\lambda)+\Div(\mu)$. In the case, where $\cG^{\cdot}=\cH^{\cdot}$
one has also
$$
	\Div(\mu\verk\lambda)=\Div(\mu)+\Div(\lambda).
$$

\begin{proposition}[\cite{knumum} Theorem 3]\label{divfunctoriality}
If $f:Y\to X$ is a morphism of noetherian schemes, $\lambda:\cF^{\cdot}\to \cG^{\cdot}$ a good
map on $X$ and for all $y\in Y$ of depth $0$ one has $f(y)\in U(\lambda)$, then
$$
	Lf^*(\lambda):Lf^*\cF^{\cdot}\to Lf^*\cG^{\cdot}
$$ 
is good on $Y$ and one has 
$$
	\Div(Lf^*(\lambda))=f^*\Div(\lambda).
$$
\end{proposition}
For more details on these functors, see \cite{knumum}.

\section{The Tamagawa number conjecture for the motive $M(F)$}

We now review the Tamagawa number conjecture for number fields in the case $s=0$, which is essentially a reformulation of the 
class number formula. As in the classical case we will reduce the main conjecture to the
case of the Tamagawa number conjecture.
The extension of the Tamagawa number conjecture 
of Bloch and Kato to coefficients is due to Kato, Fontaine-Perrin-Riou and Burns-Flach.

\subsection{\'Etale cohomology}\label{cohomology}

In this section $M$ is one of the motives $M(F)$ or $M(\eta)$.
As usual, using our fixed algebraic closure $\bar K$, we identify
continuous Galois cohomology and continuous \'etale cohomology.

In the formulation of the Tamagawa number conjecture, as well as in the sequel, we have need of several complexes of Galois cohomology, which we define following Fontaine \cite{Fontaine}.
Fix a rational prime p, and 
for every finite place $v$ of $K$, define the local unramified cohomology of $M_p$ to be
the complex
$$
	R\Gamma_f(K_v, M_p)=\begin{cases}M_p^{I_v}\stackrel{1-\Frob_v^{-1}}{\rightarrow}M_p^{I_v}&v\nmid p\\
	D_{\cris}(M_p)\stackrel{1-\phi}{\rightarrow}D_{\cris}(M_p) & v\mid p\end{cases}
$$
where $I_v$ is the inertia group at $v$. Recall that 
$D_{\cris}(M_p):=(B_{\cris}\otimes M_p)^{G_{K_v}}$ carries an action of the Frobenius of $B_{\cris}$, which is denoted by $\phi$. Moreover, the tangent space 
$(B_{\dR}/\Fil^0\otimes M_p)^{G_{K_v}}=0$ for our motive.   
This unramified cohomology is necessary to keep track of the Euler factors that 
arise when removing primes.   We further define 
$$
R\Gamma_{/f}(K_v, M_p):=\mbox{Cone}\left(R\Gamma_f(K_v,M_p)\rightarrow 
R\Gamma(K_v, M_p)\right).
$$
\begin{definition}\label{jconvention} Let $S$ be a finite set of primes of $K$ such that $M_p$ is
unramified outside of $S$ and
let $j:\Spec (\cO_K[1/pS])\hookrightarrow \Spec (\cO_K[1/p])$, 
then the \'etale sheaf $j_*M_p$ (resp. $j_*\cO_p(\eta)$ as defined in \ref{latticedefn}) 
on $\cO_K[1/p]$ 
will be denoted by $M_p$ (resp. $\cO_p(\eta)$), i.e.,
we omit $j_*$ from the notation.
\end{definition}
Using this convention, the compact support cohomology is defined 
for any Galois stable lattice $T_p\subset M_p$ as 
\begin{multline*}
R\Gamma_c(\cO_K[1/p], T_p):=\\
\mbox{Cone}\left(R\Gamma(\cO_K[1/p],T_p)\rightarrow\bigoplus_{v\mid p}R\Gamma(K_v, T_p)\oplus T_p\right)[-1].
\end{multline*}
\begin{lemma}\label{perfect}
For $K$ imaginary quadratic, $R\Gamma(\cO_K[1/p],T_p)$, $R\Gamma(K_v, T_p)$
and $R\Gamma_c(\cO_K[1/p], T_p)$ are
perfect $\cO_p$-complexes and $R\Gamma(\cO_K[1/p],T_p)$ and $R\Gamma(K_v, T_p)$ have cohomology
in degrees $0,1,2$.
\end{lemma}
\begin{proof} As $\cO_p$ is regular, this just amounts to the statement that 
the complexes have finite cohomological dimension. 
For $\cO_K[1/pS]$ this follows using
that the cohomological $p$-dimension is two because $K$ has no real place (see \cite{Milne} I.4.10). 
To show the statement for $R\Gamma_c(\cO_K[1/p], T_p)$, consider the distinguished triangle
$$
	R\Gamma_c(\cO_K[1/pS], T_p)\to R\Gamma_c(\cO_K[1/p], T_p)\to \bigoplus_{v\in S}R\Gamma(\kappa(v), T_p^{I_v})
$$
where the outer complexes are perfect by loc. cit. For $R\Gamma(\cO_K[1/p],T_p)$ we have the
distinguished triangle (using purity $i^!=i^*(-1)[-2]$)
$$
	\bigoplus_{v\in S}R\Gamma(\kappa(v),T_p(-1)^{I_p}[-2])\to R\Gamma(\cO_K[1/p],T_p)\to
	R\Gamma(\cO_K[1/pS],T_p).
$$
This proves the claim.
\end{proof}

The global unramified cohomology is defined similarly as a mapping cone 
\begin{multline*}
R\Gamma_f(\cO_K[1/p], M_p):= \\
\mbox{Cone}\left(R\Gamma(\cO_K[1/p], M_p)\rightarrow\bigoplus_{v\mid p} R\Gamma_{/f}(K_v, M_p)\right)[-1].
\end{multline*}
We have isomorphisms $H^0_f(M)\otimes_\bbQ\bbQ_p\simeq H^0_f(\cO_K[1/p], M_p)$ and thanks
to results of Soul\'e an isomorphism
$$
	H^1_f(M(1))\otimes_\bbQ\bbQ_p\simeq H^1_f(\cO_K[1/p], M_p(1))
$$ 
given by the regulator map
$$
	r_p:H^1_f(M(1))\to  H^1_f(\cO_K[1/p], M_p(1)).
$$
Further, by Artin-Verdier duality, we have that 
$$
	H^i_f(\cO_K[1/p], M_p)\simeq H^{3-i}_f(\cO_K[1/p], M_p^\vee(1))^\vee,
$$
where $^\vee$ denotes the $E_p$-dual.
Thus, we can compute $R\Gamma_f(\cO_K[1/p], M_p)$ in all degrees and get for our motives the triangle
\begin{equation}\label{Hfcomputation}
	H^0_f(\cO_K[1/p], M_p)\to R\Gamma_f(\cO_K[1/p], M_p)\to H^1_f(\cO_K[1/p], M_p^{\vee}(1))^{\vee}[-2].
\end{equation}

From the above, we deduce a fourth exact triangle
(note that $M_B\otimes_\bbQ\bbQ_p\isom M_p$):
\begin{equation}\label{galoistriangle}
R\Gamma_c(\cO_K[1/p],M_p)\rightarrow R\Gamma_f(\cO_K[1/p], M_p)\rightarrow\bigoplus_{v\mid p}R\Gamma_f(K_v,M_p)\oplus M_p.
\end{equation}
For later use, we note the behaviour of $R\Gamma_c(\cO_K[1/p],M_p)$ under addition of a finite set
of places $S$.
\begin{lemma}\label{enlargingS}
	Let  $S$ be a finite set of places of $\cO_K$ not dividing $p$, then one has a localization sequence for any $\cO_p$-lattice $T_p\subset M_p$
	$$
		R\Gamma_c(\cO_K[1/pS],T_p)\to R\Gamma_c(\cO_K[1/p],T_p)\to \bigoplus_{v\in S}R\Gamma_f(K_v,T_p).
	$$
\end{lemma}
\begin{proof}
	This follows from the localization sequence for cohomology with compact support (see \cite{Milne} II.2.3(d)) and the isomorphism
	$$
		R\Gamma(\kappa(v),M_p^{I_v})\isom R\Gamma_f(K_v,M_p)
	$$
	where $\kappa(v)$ is the residue class field and $I_v$ the inertia group at $v$.
\end{proof}

\subsection{Review of the Tamagawa number conjecture for $M$}\label{tnc}
In this section we formulate the Tamagawa number conjecture for the motives $M(F)$ and $M(\eta)$. 
Let $M$ be one of the motives $M(F)$ or $M(\eta)$.

Beilinson's regulator $r_\infty$ sits in a short exact sequence
\begin{align}\label{dirichlet}
0\rightarrow H^0_f(M)\otimes_\bbQ\bbR \to M_B\otimes_\bbQ\bbR\xrightarrow{r_\infty^\vee} H^1_f(M^\vee(1))^\vee\otimes_\bbQ\bbR\rightarrow 0.
\end{align}
Recall that $H^1_f(M(F)^\vee(1))\isom\cO_{F}^*\otimes_\bbZ E$, and that
$$
	M(F)_B\otimes_\bbQ\bbR=\bigoplus_{\tau:F\to_K\bbC}E_\infty\simeq E_\infty[G].
$$
Then $r_\infty$ is given by
\begin{align}\label{explicitreg}
	\cO_F^*\otimes_\bbZ E&\xrightarrow{r_\infty}\bigoplus_{\tau:F\to_K\bbC}E_\infty\\
	\nonumber u&\mapsto \sum_{\tau\in G}(\log |\tau(u)|)\tau,
\end{align}
where $|\tau(u)|:=({\tau(u)\overline{\tau(u)}})^{1/2}$ is the usual complex norm. Note that we
use here the Beilinson regulator, which differs from the usual normalization of the Dirichlet regulator
by a factor of $2$. This is important to get the correct $2$-part of the main conjecture.
We define the \emph{fundamental line} to be the $E$-vector space
\begin{equation}\label{fundamentalline}
\Xi(M):=\Det_E(H^0_f(M))\otimes_E
\Det_E^{-1}(H^1_f(M^\vee(1)))\otimes_E\Det_E^{-1}(M_B)
\end{equation}
By the exact sequence (\ref{dirichlet}), we have an isomorphism
$$
\vartheta_\infty:E_\infty\simeq\Xi(M)\otimes_\bbQ\bbR
$$

The leading term of the L-function at $s=0$, $L^*(M,0)$ considered as $E\otimes_\bbQ\bbC$-valued function is in $E_\infty^*$, so we can consider its image under the isomorphism above.  
\begin{conjecture}[Rational Conjecture]\label{Ratconj}
$$
\vartheta_\infty(L^*(M,0)^{-1})\in  \Xi(M)\otimes_\bbQ{1}.
$$
\end{conjecture}
The triangle in (\ref{galoistriangle}) induces an isomorphism 
\begin{equation}\label{varthetapdefn}
\vartheta_p:\Xi(M)\otimes_\bbQ\bbQ_p\stackrel{\simeq}{\rightarrow}
\Det_{E_p}R\Gamma_c(\cO_K[1/p],M_p),
\end{equation}
where one identifies $\Det_{E_p}R\Gamma_f(K_v,M_p)=E_p$.
Let $T_p$ be any $\Gal(\bar{K}/K)$-stable $\cO_p$-lattice inside of $M_p$.  In the application
to $M(F)$ we will use
$$
	T_p=\bigoplus_{\eta\in\wh G}\cO_p(\eta).
$$ 
\begin{conjecture}[Tamagawa Number Conjecture]\label{TNConj}
For all rational primes $p$, there is an equality of $\cO_p$-modules
$$
\cO_p\cdot\vartheta_p\vartheta_\infty(L^*(M,0)^{-1})=\Det_{\cO_p}R\Gamma_c(\cO_K[1/p],T_p)
$$
inside of $\Det_{E_p}R\Gamma_c(\cO_K[1/p],M_p)$, which is independent of the choice of $T_p$.
\end{conjecture}
For the independence of $T_p$, see \cite{BF1} Lemma 5. This conjecture is compatible with enlarging 
$p$ to any finite set of primes $S$ by lemma \ref{enlargingS} and hence coincides with the
usual formulation, where one uses $R\Gamma_c(\cO_K[1/pS],T_p)$.
Both conjectures hold for number fields:
\begin{theorem}\label{classnumber} Let $F$ be a number field, then the conjectures
\ref{Ratconj} and \ref{TNConj} hold for $M(F)$ and all primes $p$.
\end{theorem}
\begin{proof}  This is actually a consequence of the analytic class number formula.  For the proof 
of \ref{Ratconj} we refer to \cite{HK1} Proposition 2.3.1.  There are some differences in notation, in particular $V$ is used for the motive called $M$ in this text, and the fundamental line is denoted by $\Delta_f(V)$. The conjecture \ref{TNConj} is proved in \cite{HK1} Proposition 2.3.1 if $p\neq2$.  A proof of the case $p=2$ is given in \cite{It} 3.1.
\end{proof}
\begin{remark}
Note that for the motives $M(\eta)$ the conjecture \ref{Ratconj} is equivalent to
Stark's conjecture.
\end{remark}
 \subsection{A reformulation of the Tamagawa number conjecture}

In our proof of the equivariant main conjecture, we will not use the Tamagawa number conjecture
for the motives $M(F)$ but for certain quotients. 

Consider an abelian Galois extension $L/K$, with $K\subset F\subset L\subset \bar{K}$
and write $G_L:=\Gal(L/K)$ and $G_F:=\Gal(F/K)$. We consider $\wh{G}_F$ as a subset of
$\wh{G}_L$.
Then we have a decomposition
$$
	M(L)/M(F)\simeq 
	\bigoplus_{\eta\in \wh{G}_L \smallsetminus \wh{G}_F}M(\eta).
$$
Here we assume that $E$ contains all values of $\eta \in \wh{G}_L \smallsetminus \wh{G}_F$.
As the Tamagawa number conjecture holds for $M(L)$ and $M(F)$ 
it also holds for the quotient motive $M(L)/M(F)$ and we get from theorem \ref{classnumber}:

\begin{corollary}\label{classnumbercharacterwise}
	For all rational primes $p$, there is an equality of $\cO_p$-modules
inside of $\bigotimes_{\eta\in \wh{G}_L \smallsetminus \wh{G}_F}\Det_{E_p}R\Gamma_c(\cO_K[1/p],M(\eta)_p)$:
$$
\bigotimes_{\eta\in \wh{G}_L \smallsetminus \wh{G}_F}\cO_p\cdot\vartheta_p\vartheta_\infty(L^*(\eta,0)^{-1})=
\bigotimes_{\eta\in \wh{G}_L \smallsetminus \wh{G}_F}\Det_{\cO_p}R\Gamma_c(\cO_K[1/p],\cO_p(\eta)).
$$
\end{corollary}
We now give a reformulation of this corollary 
without using cohomology with compact support. This is necessary as the classical formulation
of the Iwasawa main conjecture also does not mention cohomology with compact support.
We first need to identify the $\cO_p$-modules given by $\Det_{\cO_p} R\Gamma_c(\cO_K[1/p],\cO_p(\eta))$.

Let $\eta\in \wh{G}_L \smallsetminus \wh{G}_F$ and $\cO_p(\eta)$ be our standard $\cO_p$-lattice inside of $M_p(\eta)$ defined in (\ref{latticedefn}). Recall that $\cO_p(\eta)^\vee $ is the $\cO_p$-dual 
of $\cO_p(\eta)$. 

\begin{proposition}[\cite{HK1} 1.2.10, \cite{It} 1.15]\label{latticecomparison} 
Consider the Artin-Verdier duality isomorphism 
$$
	\Det_{E_p} R\Gamma_c(\cO_K[1/p],M(\eta)_p)\otimes\Det_{E_p} M(\eta)_p\isom
	\Det_{E_p} R\Gamma(\cO_K[1/p],M(\eta)^{\vee}_p(1))
$$ 
and the lattice $\cO_p(\eta^{-1})^\vee\subset M(\eta)_p$.
Then, for all $p$ the $\cO_p$-structures given by
$$
	\Det_{\cO_p} R\Gamma_c(\cO_K[1/p],\cO_p(\eta^{-1})^\vee)\otimes\Det_{\cO_p} \cO_p(\eta^{-1})^\vee
$$
 on the left hand side and  by
$$
	\Det_{\cO_p} R\Gamma(\cO_K[1/p],\cO_p(\eta^{-1})(1))
$$
on the right hand side
agree under this duality isomorphism.
\end{proposition}
\begin{proof}
The statement for $p\neq 2$ is \cite{HK1} 1.2.10 applied to $T_p=\cO_p(\eta^{-1})^\vee$. 
The statement for $p=2$ follows from \cite{It} 1.15 using that 
$R\Gamma(\cO_K[1/p],\cO_p(\eta^{-1})^\vee(1))$ is concentrated in degrees $\le 2$ and
that $\wh{H}^0(\bbR,\cO_2(\eta^{-1})^\vee)=0$, which gives $\Det_{\cO_2}\wh{H}^0(\bbR,\cO_2(\eta^{-1})^\vee)=\cO_2$.
\end{proof}

\begin{definition}\label{zetaelementdefn} Let $\eta$ be non-trivial, so that
$H^0_f(M(\eta))=0$, and consider the lattice $\cO(\eta^{-1})^\vee\subset M(\eta)_B$ with
generator  $t_B(\eta^{-1})^\vee=t_B(\eta)$ from \ref{latticedefn}. 
Then there is a unique $ z(\eta)\in H^1_f(M(\eta)^{\vee}(1))\otimes_{\bbQ}\bbR$, the \emph{zeta element} of $M(\eta)$,
such that
$$
	\vartheta_\infty(L^*(\eta,0)^{-1})=z(\eta)^{-1}\otimes (t_B(\eta^{-1})^\vee)^{-1}
$$
in $(\Det_{E_\infty}^{-1}H^1_f(M(\eta)^{\vee}(1))\otimes_{\bbQ}\bbR)\otimes_\bbR (\Det_{E_\infty}^{-1}M(\eta)_B\otimes_{\bbQ}\bbR)$. Note that
$z(\eta)$ depends on the choice of $t_B(\eta^{-1})^\vee$.
We also let 
$$
	z_{p}(\eta):=\prod_{\frp\mid p}(1-\eta(\frp))z(\eta)
$$
(product over places $\frp\mid p$ of $K$)
be the zeta element with the Euler factors above $p$ at $s=0$ removed (here we use the
convention that $\eta(\frp)=0$, if $\eta$ is ramified at $\frp$). 
\end{definition}
Consider the regulator map for $\eta\in \wh{G}_L \smallsetminus \wh{G}_F$
$$
	r_p: H^1_f(M(\eta)^\vee(1))\otimes_E E_p\to H^1_f(\cO_K[1/p],M(\eta)_p^{\vee}(1)).
$$
From theorem \ref{classnumber}
we see that the sum of the zeta elements 
$$
	\bigoplus_{\eta\in\wh{G}_L \smallsetminus \wh{G}_F}z(\eta)\in H^1_f(\bigoplus_{\eta\in\wh{G}_L \smallsetminus \wh{G}_F} M(\eta)^{\vee}(1))\otimes_{\bbQ}\bbR
$$
is in fact contained in $H^1_f(\bigoplus_{\eta\in\wh{G}_L \smallsetminus \wh{G}_F} M(\eta)^{\vee}(1))$ 
(note that we do not know this for the individual $z(\eta)$).
In particular, we can consider the element
$$
	r_p(\bigoplus_{\eta\in\wh{G}_L \smallsetminus \wh{G}_F }z(\eta))\in H^1_f(\cO_K[1/p],\bigoplus_{\eta\in\wh{G}_L \smallsetminus \wh{G}_F } M(\eta)_p^{\vee}(1)).
$$
We want to understand, the $\cO_p$-submodule in $H^1(\cO_K[1/p], M(\eta)_p^{\vee}(1))$ (without
the $_f$)
generated by these elements.
Taking determinants in the above identity,
the sums become tensor products and we get the following reformulation of 
corollary \ref{classnumbercharacterwise}:
\begin{corollary}\label{classnumbercor}
The $\cO_p$-module
$$
	\bigotimes_\eta r_p(z_{p}(\eta)\cO_p)\in \bigotimes_\eta \Det_{E_p}R\Gamma(\cO_K[1/p],M(\eta)_p^{\vee}(1))[1],
$$
where the tensor product is taken over all $\eta\in \wh{G}_L \smallsetminus \wh{G}_F$, coincides with the
$\cO_p$-module
$$
	\bigotimes_\eta \Det_{\cO_p}R\Gamma(\cO_K[1/p],\cO_p(\eta^{-1})(1))[1].
$$
\end{corollary}
\begin{proof}
	By  proposition \ref{latticecomparison} the statement
	in corollary \ref{classnumbercharacterwise} is equivalent to the statement that 
	under the isomorphism $\vartheta_p$ the $\cO_p$-module 
	$$
		\bigotimes_{\eta\in\wh{G}_L \smallsetminus \wh{G}_F }
		\Det_{\cO_p}^{-1}(\cO_pr_p(z(\eta))\otimes \cO_p(\eta^{-1})^\vee)\otimes
		\bigotimes_{\frp\mid p}\Det_{\cO_p}^{-1}R\Gamma_f(K_\frp,\cO_p(\eta^{-1})^\vee)
	$$
	coincides with
	$$
		\bigotimes_{\eta\in\wh{G}_L \smallsetminus \wh{G}_F }\Det_{\cO_p}^{-1} R\Gamma(\cO_K[1/p],\cO_p(\eta^{-1})(1))\otimes \Det_{\cO_p}^{-1} \cO_p(\eta^{-1})^\vee
	$$
	The claim follows from the fact that 
	$$
		\Det_{\cO_p}\bigoplus_{\frp\mid p}R\Gamma_f(K_\frp,\cO_p(\eta^{-1})^\vee)=\prod_{\frp\mid p}(1-\eta(\frp))
	$$
	(see \cite{HK1} 1.2.5).
\end{proof}

\section{Review of the Euler System of elliptic units}\label{euler-system}
In the proof of the Iwasawa main conjecture, the machinery of Euler systems is an essential tool.  In this section, we construct an Euler system  by twisting the elliptic units by a finite order character. 
The general theory of Euler systems, invented by Kolyvagin, was further developed
by Kato, Perrin-Riou and Rubin (alphabetical order). We follow
Rubin as his approach is closest to our setting.

 \subsection{Euler systems}\label{es-definition}
Rubin gives a general definition for an Euler system in \cite{Rubin2}.  We recall this definition using much of his notation. Fix a prime $p$ and let $T_p$ be a $p$-adic representation of the absolute Galois group of $K$ with coefficients in  $\cO_p$, and let $\cN$ denote an ideal of $\cO_K$
divisible by $p$ and the primes at which $T_p$ is ramified. 
Denote by  $\cK:=\bigcup_{(\frq, \cN_0)=1}K(\frq)$ the union of the ray class fields of conductor
prime to the prime to $p$-part $\cN_0$ of $\cN$. We denote by $K_\infty$ the maximal abelian $\bbZ_p$-extension of $K$ unramified outside of $p$. 
Note that no finite prime of $\cO_K$ splits completely in $K_\infty$ and $\Gal(K_\infty/K)\simeq\bbZ_p^2$.

\begin{definition}[\cite{Rubin2} Definition 2.1.1 and 2.1.3]\label{rubin-defn}
A collection of Galois cohomology classes $c_\frm\in H^1(K(\frm)\cap\cK,T_p)$ for all ideals $\frm$ of $\cO_K$
is called an {\em Euler system} for $(\cK,T_p,\cN)$ if for every prime ideal $\frq$
$$
\Cor_{K(\frm\frq)\cap\cK/K(\frm)\cap\cK}(c_{\frm\frq})=\begin{cases}
P(\Frob_\frq^{-1}|T_p^\vee(1);\Frob_\frq^{-1})c_\frm &\frq\nmid\frm\cN\\
c_\frm&\frq\mid\frm\cN.
                                         \end{cases}
$$
Here the Euler factors are given by the characteristic polynomial
$$
P(\Frob_\frq^{-1}|T_p^\vee(1);x)=\det(1-\Frob_\frq^{-1}x|\Hom_{\cO_p}(T_p,\cO_p(1)))\in\cO_p[x].
$$
\end{definition}

\subsection{Elliptic Units}\label{elliptic-units}
We recall the definition of the Euler system of elliptic units, following the treatment of 
Kato \cite{Kato} section 15.

First we recall Kato's definition of a CM-pair $(E, \alpha)$ of modulus $\frm$. Fix a non-zero
ideal $\frm$ of $\cO_K$,  such that $\cO_K^*\to (\cO_K/\frm)^*$ is injective.
Then a CM-pair $(E,\alpha)$ consists of an elliptic curve $E/K'$, where $K'/K$ is
a field extension together with an isomorphism $\cO_K\simeq\End(E)$, such that
the composition $\cO_K\simeq\End(E)\to \End_{K'}(\Lie(E))=K'$ is the canonical inclusion,
and $\alpha\in E(K')$ is a torsion point, such that the annihilator of $\alpha$ in $\cO_K$
coincides with $\frm$. Any isomorphism between two CM-pairs of modulus $\frm$ over $K'$ is
unique because $\cO_K^*\to (\cO_K/\frm)^*$ is injective.

The main theorem of complex multiplication implies that there exists a CM-pair 
(unique up to unique isomorphism) of modulus $\frm$ over the ray class field $K(\frm)$
which is isomorphic to $(\bbC/\frm,1\mbox{ mod }\frm)$ over $\bbC$.

Kato constructs in \cite{Kato} 15.4 for each $\fra\subset \cO_K$ which is prime to $6$ a
function $_\fra\theta_E\in \cO(E \smallsetminus E[\fra])^*$,
which is characterized uniquely by the following two properties (denote by $E[\fra]$ the subgroup
of elements annihilated by $\fra$):
\begin{itemize}
\item The divisor of $_\fra\theta_E$ is $\N(\fra)(0)-\sum_{P\in E[\fra]}(P)$
\item For each integer $b$, which is prime to $\fra$, 
one has $[b]_*{_\fra}\theta_E={_\fra}\theta_E$, where $[b]_*$ is the norm with respect to
the isogeny $[b]$.
\end{itemize}
We can now define elliptic units following Kato:
\begin{definition}\label{zetadefinition}
	Fix a prime $p$ and choose an integer $r\ge 1$, such
	that $\cO_K^*\hookrightarrow (\cO_K/ p^r)^*$ is injective. Let $\fra$ be prime to
	$6p$. For any non zero ideal $\frm$
	of $\cO_K$ prime to $\fra$ we define
	$$
		\zeta_\frm:={_\fra\zeta_\frm}:=\N_{K(p^r\frm)/K(\frm)} {_\fra \theta_E(\alpha)^{-1}}\in {K(\frm)}^*
	$$
	and if $K\subset L\subset K(\frm)$ has conductor $\frm$
	$$
		\zeta_L:=\N_{K(\frm)/L}\zeta_\frm.
	$$
	Here $(E,\alpha)$ is ``the" CM-pair of modulus $p^r\frm$ defined over $K(p^r\frm)$. Note that this is independent of
	the chosen $r\ge 1$. We omit the auxiliary ideal $\fra$ from the notation, whenever no confusion is possible.
\end{definition}
These elements have the following properties.
\begin{proposition}\label{unitproperties} Let $p^r$ and $\fra$ be as in definition
\ref{zetadefinition}, then:
\begin{enumerate}
\item (Integrality) $\zeta_{\frm}\in \cO_{K(\frm)}^*$ if $p^r\frm$ is divisible by two different primes and
$\zeta_{\frp^n}\in \cO_{K(\frp^n)}[1/\frp]^*$ if $p^r\frm$ is a power of $\frp$.
\item (Euler system property) For a prime ideal $\frq\subset \cO_K$  such that $\frm\frq$ is 
prime to $\fra$ one has
$$
\N_{K(\frq\frm)/K(\frm)}(\zeta_{\frq\frm})=
\begin{cases} 
\zeta_{\frm}^{1-\Frob_\frq^{-1}} & \frq\nmid p\frm\\
\zeta_{\frm}  & \frq\mid p\frm \\
\end{cases}
$$
\item (Independence from $\fra$) If $\fra,\frb\subset \cO_K$ are prime to $6p$ and
$\frm$ is prime to $\fra\frb$ let $\sigma_\fra=(\fra,K(\frm)/K)$
and $\sigma_\frb=(\frb,K(\frm)/K)$ be the Artin symbols in $G(\frm)$, then
$$
	_\frb \zeta_{\frm}^{\N(\fra)-\sigma_\fra}=
	{_\fra}\zeta_{\frm}^{\N(\frb)-\sigma_\frb}.
$$
\item (Relation to $L$-values) For any non-trivial character $\eta:G(\frm)\to \bbC^*$ (not necessarily proper) we have 
$$
	\sum_{\tau \in G(\frm)}\eta(\tau)\log|\tau({_\fra}\zeta_{\frm})|=
	(\N(\fra)-\eta(\fra))
	\lim_{s\to 0}s^{-1}L_{p\frm}(\eta,s),
$$
where $|z|=(z\bar z)^{1/2}$ (note that $\eta(\fra)=\eta^{-1}(\sigma_\fra)$ by our normalization of the reciprocity map).
\end{enumerate}
\end{proposition}
\begin{proof}
Observe first that the function ${_\fra}\theta_E$ is uniquely determined by the
norm compatibility and its divisor. Then it is clear that ${_\fra}\theta_E$ is a twelfth
root of the function in Chapter II of \cite{deShalit}. Property
1) follows immediately from \cite{deShalit} II. 2.4. and property 2) follows in the same way as
II. 2.5. i) in \cite{deShalit}, if one observes that $w_{p^r\frm}=w_{p^r\frm\frq}=1$ in our case.
Property 3) is \cite{Kato} 15.4.4 and property 4) is \cite{Kato} (15.5.1).
\end{proof}

\begin{corollary}\label{eulerdefn} Choose $\fra$ prime to $6p$ and let $\cK:=\bigcup_{(\frq, \fra)=1}K(\frq)$. 
Then the $_\fra\zeta_\frm\in \cO_{K(\frm)}[1/p]^*\otimes_\bbZ\bbZ_p\subset H^1(\cO_{K(\frm)}[1/p],\bbZ_p(1))$ 
for all $\frm$ prime to $\fra$ form
an Euler system for $(\cK,\bbZ_p(1),p\fra)$ in the sense of definition \ref{rubin-defn}.
\end{corollary}

\subsection{The twisted Euler system}\label{twisted-euler}
Consider a character
$$
	\eta:G(\frf_\eta)\to E^*
$$
of conductor $\frf_\eta$.
Let $\cK$ be the field extension defined in \ref{eulerdefn} 
and  assume that $\fra$ is chosen prime to $\frf_\eta$.
We wish to study a twist of the Euler system of elliptic units by $\eta$.

Consider the composition of the following two maps (\ref{ttwist}) and (\ref{trace})
\begin{align}\label{ttwist}
H^1(\cO_{K(\frf_\eta\frm)}[1/p],\cO_p(1))\xrightarrow{\otimes t_p(\eta)}H^1(\cO_{K(\frf_\eta\frm)}[1/p],\cO_p(\eta)(1)),
\end{align}
where we have identified
$$
	H^1(\cO_{K(\frf_\eta\frm)}[1/p],\cO_p(1))\otimes_{\cO_p} \cO_p(\eta)\simeq H^1(\cO_{K(\frf_\eta\frm)}[1/p],\cO_p(\eta)(1)),
$$
and of the trace map (for $\cO_{K(\frf_\eta\frm)}[1/p]\to \cO_{K(\frm)}[1/p]$)
\begin{align}\label{trace}
H^1(\cO_{K(\frf_\eta\frm)}[1/p],\cO_p(\eta)(1))&\xrightarrow{\tr_{K(\frf_\eta\frm)/K(\frm)}}
H^1(\cO_{K(\frm)}[1/p],\cO_p(\eta)(1)).
\end{align}

\begin{definition}\label{twisted-euler-defn}
For all $\frm$ prime to $\fra$ define
$$
	\zeta_\frm(\eta):={_\fra}\zeta_\frm(\eta):=\tr_{K(\frf_\eta\frm)/K(\frm)} (\zeta_{\frf_\eta\frm}\otimes t_p(\eta))
\in H^1(\cO_{K(\frm)}[1/p],\cO_p(\eta)(1)).
$$
For any field $K\subset F\subset K(\frm)$ of conductor $\frm$, we define
$$
	\zeta_F(\eta):= \tr_{K(\frm)/F}\zeta_\frm(\eta).
$$
\end{definition}
Note that $\zeta_F(\eta)$ depends on $t_p(\eta)$.

The following proposition is shown in Rubin \cite{Rubin2}
\begin{proposition}[\cite{Rubin2} 2.4.2] Let $\cK$ be as above and $\fra$ prime to $6p$. The collection
$$
_\fra\zeta_\frm(\eta)\in H^1(\cO_{K(\frm)}[1/p],\cO_p(\eta)(1))
$$
for all ideals $\frm$ prime to $\fra$ is an Euler system for $(\cK,\cO_p(\eta)(1),p\frf_\eta\fra)$.
\end{proposition}
\subsection{A compatibility}
Suppose that $K\subset F\subset L\subset \cK$ with $G:=\Gal(L/F)$ and let $H$ be the quotient so that
$$
0\to G\to \Gal(L/K)\to H\to 0.
$$
For later use we need a compatibility between
$\zeta_F(\eta)$ as  in definition \ref{twisted-euler-defn} and $p_{\eta^{-1}}(\zeta_{L})$.
Consider the following diagram:
$$
\begin{CD}
\Spec L@>\wt j>>\Spec \cO_{L}[1/p]\\
@V\pi VV@V\wt \pi VV\\
\Spec F@>j>>\Spec \cO_{F}[1/p].
\end{CD}
$$
Then we have $\wt \pi_*\wt j_*=j_* \pi_*$ and
$$
	H^1(\cO_{L}[1/p],\wt j_*\cO_p(1))\simeq H^1(\cO_{F}[1/p], j_*\pi_*\cO_p(1)).
$$
We can identify $\pi_*\cO_p(1)\simeq \cO_p[G](1)$, which we consider as $\cO_p(1)$-valued maps on $G$. 
Consider the map induced by $p_{\eta^{-1}}:\cO_p[G]\to \cO_p(\eta)$ 
$$
	p_{\eta^{-1}}:H^1(\cO_{F}[1/p],j_*\cO_p[G](1))\to H^1(\cO_{F}[1/p],j_*\cO_p(\eta)(1)).
$$
We return now to our convention, to omit $j_*$ from the notation.
\begin{lemma}\label{zetaidentification}
 The image of $\zeta_{L}\in H^1(\cO_{L}[1/p],\cO_p(1))$ under the above map
$p_{\eta^{-1}}$
coincides with $\zeta_F(\eta)$.
\end{lemma}
\begin{proof} We write $G=\Gal(L/F)$. The corestriction 
$$
\tr_{L/F}:H^1(\cO_{L}[1/p],\cO_p(\eta)(1))\to H^1(\cO_{F}[1/p],\cO_p(\eta)(1))
$$ 
is induced by the map
$\cO_p[G]\otimes_{\cO_p}\cO_p(\eta)\to \cO_p(\eta)$ given by
$$
	f\otimes t_p(\eta)\mapsto \sum_{g\in G} g(f(g^{-1})t_p(\eta)).
$$
To show the lemma it suffices to show the commutativity of the diagram
	$$
		\begin{xy}
			\xymatrix{\cO_p[G] \ar[r]^{p_{\eta^{-1}}}\ar[d]_{\otimes t_p(\eta)}&\cO_p(\eta)  \\
			\cO_p[G]\otimes_{\cO_p}\cO_p(\eta)\ar[ru]}\\
		\end{xy}
	$$
	where the diagonal map is given by the above formula.
Writing $f=\sum_{g\in G}f(g)\delta_g$ where $\delta_g$ is the delta function at $g\in G$, we get 
	$$
		p_{\eta^{-1}}f=\sum_{g\in G}f(g)p_{\eta^{-1}}\delta_g=(\sum_{g\in G}\eta^{-1}(g)f(g))t_p(\eta).
	$$
On the other hand
$$
	\sum_{g\in G} g(f(g^{-1}) t_p(\eta))=(\sum_{g\in G} \eta(g)f(g^{-1}))t_p(\eta).
$$
\end{proof}
\subsection{Relation to zeta elements}\label{zetaelements}
In this section
we make the relation between the Euler system and the zeta elements precise.
This is crucial for the reduction of the main conjecture to the Tamagawa number conjecture.

Let $K_\infty= \bigcup_{n\ge 0} K_n$ be the maximal $\bbZ_p^2$-extension of $K$ which
is unramified outside of $p$ and where $K\subset K_n\subset K_\infty$ is the unique
subextension with Galois group $(\bbZ/p^n\bbZ)^2$.

Fix an integral ideal $\frf_\chi\subset \cO_K$ (which will later be the conductor of a character).
\begin{definition}\label{leveldefn}
Let $\eta:G(\frf_\chi)\to E^*$ be a character of conductor $\frf_\eta$.
The biggest $n\ge 0$, such that $K_n\subset K(\frf_\eta)$ is called the \emph{level } of $\eta$.
\end{definition}
Observe that if the level of $\eta$ is big enough, then $\eta$ is ramified at
all primes above $p$, so that $\eta(\frp)=0$ and  $z_p(\eta)=z(\eta)$.

Our aim in this section is to show that for characters $\eta$ of big enough level 
$\zeta_K(\eta^{-1})$ as in definition \ref{twisted-euler-defn} essentially coincides with the zeta element $z_p(\eta)$.
Note that 
$$
H^1_f( M(\eta^{-1})(1))=p_{\eta}(\cO_{K(\frf_\eta)}^*\otimes_\bbZ E)=
p_{\eta}(\cO_{K(\frf_\chi)}^*\otimes_\bbZ E)
$$
because the motive $M(\eta^{-1})$ does not depend on the group where $\eta$ is considered
(cf. equation \eqref{motivdec}). By definition and lemma \ref{zetaidentification} (for $L=K(\frf_\eta)$
and $F=K$)
$$
	\zeta_K(\eta^{-1})=\tr_{K(\frf_\eta)/K}(\zeta_{\frf_\eta}\otimes t_p(\eta^{-1}))=p_\eta(\zeta_{\frf_\eta})\in p_\eta(\cO_{K(\frf_\eta)}[1/p]^*\otimes_\bbZ E).
$$
\begin{theorem}\label{zetaexplicit}  Let
$\eta:G(\frf_\chi)\to E^* $ be a character of conductor $\frf_\eta$ and level $n$ such that $\eta$ is ramified at the primes above $p$.
\begin{itemize}
\item [a)] Consider $\zeta_K(\eta^{-1})$ as an element in $p_\eta(\cO_{K(\frf_\eta)}[1/p]^*\otimes_\bbZ E)$.
Then $\zeta_K(\eta^{-1})\in p_{\eta}(\cO_{K(\frf_\eta)}^*\otimes_\bbZ E)=H^1_f( M(\eta^{-1})(1))$
and
$$
\zeta_K(\eta^{-1})=(\N\fra-\eta(\fra))z_{p}(\eta).
$$
\item [b)] For each ideal $\fra\neq \cO_K$ prime to $6p\frf_\eta$, one has $\N\fra\neq \eta(\fra)$ 
as automorphisms of $M(\eta)_p$, so that $\N\fra-\eta(\fra)$ is invertible.
\item [c)] Consider the regulator
$$
	r_p:H^1_f(M(\eta^{-1})(1))\to H^1(\cO_K[1/p],M(\eta^{-1})_p(1))
$$
and let $\fra\neq \cO_K$ be prime to $6p\frf_\eta$. Then 
$$
	r_p(z_{p}(\eta)\cO_p)=(\N\fra-\eta(\fra))^{-1}\zeta_K(\eta^{-1})\cO_p\subset H^1(\cO_K[1/p],M(\eta^{-1})_p(1)).
$$
\end{itemize}
In particular, the element 
$\zeta_K(\eta^{-1})$ is not torsion in $ H^1(\cO_K[1/p],\cO_p(\eta^{-1})(1))$.
\end{theorem}

\begin{proof} a) Recall that 
$z(\eta)\in H^1_f(M(\eta)^\vee(1))\otimes_\bbQ\bbR\isom H^1_f(M(\eta^{-1})(1))\otimes_\bbQ\bbR$.
By definition the element $z(\eta)\otimes t_B(\eta^{-1})^\vee$ is the one which maps
to $L^*(\eta,0)$ under $\vartheta_\infty^{-1}$. Consider 
$\zeta_K(\eta^{-1})\in p_{\eta}(\cO_{K(\frf_\eta)}[1/p]^*\otimes_\bbZ E)$.
We first show that 
$$
	p_{\eta}(\zeta_{\frf_\eta})\in p_{\eta}(\cO_{K(\frf_\eta)}^*\otimes_\bbZ E)=H^1_f(M(\eta^{-1})(1)).
$$
if  $\eta$ is ramified at all places above $p$. If $\frf_\eta$ is divisible by
at least two different primes, then by \ref{unitproperties}, $\zeta_{\frf_\eta}$ is already in $\cO_{K(\frf_\eta)}^*\otimes_\bbZ E$.
By our assumption $p^k\mid \frf_\eta$
for some $k\ge 1$. This implies that $\zeta_{\frf_\eta}$ can be only a non-unit if
$\frf=\frp^l$ for $\frp$ the only prime above $p$.
Consider the exact sequence
$$
	1\to \cO_{K(\frf_\eta)}^*\otimes_{\bbZ}E\to \cO_{K(\frf_\eta)}[1/\frp]^*\otimes_{\bbZ}E\to 
	\prod_{v\mid\frp}E
$$
(product over all places $v\mid \frp$ of $K(\frf_\eta)$). As a $G$-module the 
product $\prod_{v\mid\frp}E\simeq E[G(\frf_\eta)/D_v]$, where $D_v$ is the decomposition group 
at $v$. The fact that $\eta$ is ramified at $v$ implies
$$
	p_{\eta}\prod_{v\mid\frp}E=0
$$
and the above claim.
To show the assertion in a) it suffices to compute the regulator $r_\infty$ because
this is an injective map.
 With the explicit form of the
regulator $r_\infty$ in (\ref{explicitreg}) we get using \ref{unitproperties} (4):
\begin{align*}
	r_\infty(\zeta_K(\eta^{-1}))&=p_{\eta}\sum_{\tau\in G(\frf_\eta)}(\log\mid \tau(\zeta_{\frf_\eta})\mid)\tau\\
				&=\sum_{\tau\in G(\frf_\eta)}\eta(\tau)(\log\mid \tau(\zeta_{\frf})\mid)t_B(\eta^{-1})\\
				&=(\N\fra-\eta(\fra))\lim_{s\to 0}s^{-1}L_{p\frf_\eta}(\eta,s)t_B(\eta^{-1}).
\end{align*}
As $p$ divides $\frf_\eta$ we get
$L_{p\frf_\eta}(\eta,s)=L_{\frf_\eta}(\eta,s)=L(\eta,s)$ and $z(\eta)=z_p(\eta)$. Thus, $r_\infty(\zeta_K(\eta^{-1}))=(\N\fra-\eta(\fra))L^*(\eta,0)t_B(\eta^{-1})$. On the other hand 
$$
	r_\infty(z(\eta))=L^*(\eta,0)t_B(\eta^{-1}),
$$
so that $\zeta_K(\eta^{-1})=(\N\fra-\eta(\fra))z(\eta)=(\N\fra-\eta(\fra))z_p(\eta)$ as $r_\infty$
is injective. This implies a).

b) This is clear as $\eta(\fra)$ is a root of unity and $\N\fra$ is not.

c) From a) and the choice of $\fra$ we get
$$
	r_p(\zeta_K(\eta^{-1}))\cO_p=(\N\fra-\eta(\fra))r_p(z_p(\eta))\cO_p.
$$
\end{proof}
Let $K_n(\frf_\chi):=K_nK(\frf_\chi) $ be the
compositum of $K_n$ and $K(\frf_\chi)$ and write $G_n(\frf_\chi):=\Gal(K_n(\frf_\chi)/K)$.

Combining the above theorem \ref{zetaexplicit} with corollary \ref{classnumbercor} for $L=K_n(\frf_\chi)$ and $F=K_{n-1}(\frf_\chi)$ 
one gets:

\begin{corollary}\label{generationbyEuler} Let  $n$ be so big
that all $\eta\in\wh{G}_n(\frf_\chi) \smallsetminus \wh{G}_{n-1}(\frf_\chi) $ are ramified at all primes above $p$. Let $\fra$ be as in theorem \ref{zetaexplicit}, then
the $\cO_p$-module
$$
	\bigotimes_\eta (\N\fra-\eta(\fra))^{-1}\zeta_K(\eta^{-1})\cO_p\subset \bigotimes_\eta \Det_{E_p}^{-1}R\Gamma(\cO_K[1/p],M(\eta)^\vee_p(1)),
$$
where the tensor product is taken over all $\eta\in \wh{G}_n(\frf_\chi) \smallsetminus \wh{G}_{n-1}(\frf_\chi)$,
coincides with the
$\cO_p$-module
$$
	\bigotimes_\eta \Det_{\cO_p}^{-1}R\Gamma(\cO_K[1/p],\cO_p(\eta^{-1})(1)).
$$
\end{corollary}

\section{Iwasawa modules}

In this section we introduce the basic Iwasawa modules we want to study
and state some of their properties used later.

\subsection{The Iwasawa algebras $\Lambda$ and $\Omega$}

Consider inside $K(p^\infty):= \bigcup_{n\ge 1}K(p^n)$ the 
maximal $\bbZ_p^2$-extension $K_\infty$ of $K$, so that 
$$
\Gamma:=\Gal(K_\infty/K)\simeq\bbZ_p^2.
$$
We denote by $K\subset K_n\subset K_\infty$ the unique subextension with
Galois group $G_n:=\bbZ_p^2/p^n\bbZ_p^2$. 
For an ideal $0\neq \frf\subset \cO_K$ we define
\begin{equation}\label{Gfdefn}
\cG_\frf:=\Gal(K(\frf p^\infty)/K).
\end{equation}
We denote by $\Delta\subset \cG_\frf$ the torsion subgroup and fix once for all a splitting
$$
	\cG_\frf\simeq \Delta\times \Gamma.
$$

For each profinite group $\cG=\prolim \cG/\cH$ we define its \emph{Iwasawa algebra} to be
the inverse limit
$$
	\Lambda(\cG):=\prolim_{\cH\subset \cG}\bbZ_p[\cG/\cH].
$$
Two Iwasawa algebras are especially important in the sequel:
\begin{definition}\label{Iwasawa-alg-def}The \emph{Iwasawa algebra} for $\Gamma$ is denoted
by
$$
\Lambda:=\Lambda(\Gamma),
$$
which is (non-canonically) isomorphic to $\bbZ_p[[T,S]]$. The Iwasawa algebra for $\cG_\frf$ is
denoted by
$$
\Omega:=\Lambda(\cG_\frf),
$$
which is (non-canonically) isomorphic to 
$$
\Omega\isom \bbZ_p[\Delta][[T,S]].
$$
We also let $\Lambda_\cO:=\Lambda\wh\otimes_{\bbZ_p}\cO_p$ and $\Omega_\cO:=\Omega\wh\otimes_{\bbZ_p}\cO_p$
be the Iwasawa algebras with coefficients in $\cO_p$.
\end{definition}
Both Iwasawa algebras $\Lambda$ and $\Omega$ carry a natural action of $\Gal(\bar{K}/K)$,
which acts through its quotient $\Gamma$ (resp. $\cG_\frf$) by the canonical inclusions
$\Gamma\subset \Lambda^*$ (resp. $\cG_\frf\subset \Omega^*$). 
The $\Gal(\bar{K}/K)$-module $\Lambda$ is unramified outside of $p$ and $\Omega$ is unramified outside of
$\frf p$. 
Note that $\Lambda$ and $\Omega$ are products of local rings so that we can apply the
Nakayama lemma to each component of $\Lambda$ and $\Omega$.

\subsection{The basic Iwasawa modules}
Fix an integral ideal $\frf$ and let $\Omega=\Lambda(\cG_\frf)$. Let $\eta$ be a character of conductor $\frf_\eta\mid \frf$ and $\cO_p(\eta)$ the associated $\cO_p$-module
with $\Gal(\bar{K}/K)$-action by $\eta$ as defined in \ref{latticedefn}.
The action on  $\cO_p(\eta)$ is unramified outside of $p\frf_\eta$ and factors through
$\cG_{\frf}$. In particular, $\cO_p(\eta)$ is an $\Omega_\cO$-module and by restriction
also a $\Lambda_\cO$-module.

Recall that we consider $\cO_p(\eta)$ as \'etale sheaf on $\cO_K[1/p]$ via the map $j:\Spec (\cO_K[1/p\frf_\eta])\hookrightarrow \Spec (\cO_K[1/p])$ and that  we omit $j_*$ from the notation.
\begin{definition}
For the Iwasawa algebra $\Lambda$  let
$$
\Lambda(\eta):=\cO_p(\eta)\otimes_{\bbZ_p}\Lambda
$$
considered as \`etale sheaf (of $\Lambda_\cO$-modules) on $\Spec (\cO_K[1/p])$.
We also use the notation $\Lambda(\eta)(1):=\Lambda(\eta)\otimes_{\bbZ_p}\bbZ_p(1)$ and
$\Omega(1):=\Omega\otimes_{\bbZ_p}\bbZ_p(1)$.
\end{definition}
We have
\begin{equation}
H^i(\cO_K[1/p],\Lambda(\eta))=\prolim_{K\subset K_n\subset K_\infty}H^i(\cO_{K_n}[1/p],\cO_p(\eta))
\end{equation}
and
\begin{equation}
H^i(\cO_K[1/p\frf],\Omega(1))=\prolim_{K\subset F\subset K(p^\infty\frf)}H^i(\cO_{F}[1/p\frf],\bbZ_p(1)).
\end{equation}
In particular,
\begin{equation}\label{h0vanishing}
H^0(\cO_K[1/p\frf],\Omega(1))=0=H^0(\cO_K[1/p],\Lambda(\eta)(1)).
\end{equation}
Here the (left) $\Lambda_{\cO}$-module structure on $H^i(\cO_K[1/p],\Lambda(\eta))$ 
is induced by multiplication with the inverse on $\Lambda$, so that 
$\gamma\in \Gamma$ acts on $\Lambda(\eta)$ via $\eta^{-1}(\gamma)$
(see \cite{Ho-Ki} Appendix B for details).
We consider also the cohomology with
compact support
\begin{equation}
H^i_c(\cO_K[1/p],\Lambda(\eta))
\end{equation}
and the local cohomology groups 
\begin{equation}
H^i(K_v,\Lambda(\eta))
\end{equation}
and similarly for $\Omega(1)$.
These $\Lambda_\cO$-modules (resp. $\Omega$-modules)
are the basic Iwasawa modules, which are involved in the formulation of the main
conjecture. 

We collect some information about these Iwasawa modules. The following lemma
will be often used without further comment.

\begin{lemma} \label{perfectness}
Let $\Lambda_\cO$ be the basic Iwasawa algebras with coefficients in $\cO_p$.
Then the complexes of $\Lambda_\cO$-modules 
$$
R\Gamma(\cO_K[1/p],\Lambda( \eta)(1)),\  R\Gamma_c(\cO_K[1/p],\Lambda( \eta)(1))
\mbox{ and } R\Gamma(K_v,\Lambda( \eta)(1))
$$
are perfect. The same statement holds for the complexes of $\Omega$-modules
$$
R\Gamma(\cO_K[1/p\frf],\Omega(1)),\ 
R\Gamma_c(\cO_K[1/p\frf],\Omega(1))\mbox{ and }R\Gamma(K_v,\Omega(1)).
$$ 
For all primes
$\frl\nmid p$ of $K$ one has 
$$
	\Det_{\Lambda_\cO}R\Gamma_{\kappa(\frl)}(\cO_{K_\frl},\Lambda( \eta)(1))=\begin{cases}
								(1-\eta(\frl)\Frob^{-1}_\frl)^{-1}\Lambda_\cO&\mbox{ if }\eta \mbox{ is unramified at }\frl\\
								\Lambda_\cO&\mbox{ else}
	                                                                         \end{cases}
$$
where $\Frob^{-1}_\frl\in \Gal(K_\infty/K)$ is the inverse of Frobenius at $\frl$,
$\kappa(\frl)$ is the residue field at $\frl$ and $R\Gamma_{\kappa(\frl)}(\cO_{K_\frl},\Lambda( \eta)(1))$
is the complex which computes the cohomology with support in $\kappa(\frl)$.
\end{lemma}
\begin{proof}
The first statement and the second statement for  $\cO_K[1/p\frf]$ follow from \cite{Fu-Ka} 1.6.5 (2).
Using the localization sequence 
$$
	\bigoplus_{\frl\mid\frf,\frl\nmid p}R\Gamma_{\kappa(\frl)}(\cO_{K_\frl},\Lambda( \eta)(1))\to
	R\Gamma(\cO_K[1/p],\Lambda( \eta)(1))\to R\Gamma(\cO_K[1/p\frf],\Lambda( \eta)(1))
$$
it suffices to consider $R\Gamma_{\kappa(\frl)}(\cO_{K_\frl},\Lambda( \eta)(1))$,
where $\kappa(\frl)$ is the residue field of $\frl$.
By purity this is isomorphic to
$R\Gamma(\kappa(\frl),\Lambda( \eta)^{I_\frl})[-2]$, where $I_\frl$ is the inertia group. 
If $\eta$ is ramified at $\frl$, $I_\frl$ acts non trivially and
one gets $\Lambda( \eta)^{I_\frl}=0$. Otherwise the complex is represented by
$(\Lambda_\cO\xrightarrow{1-\eta(\frl)\Frob^{-1}_\frl}\Lambda_\cO)[-2]$, which is obviously perfect and
has the right determinant.
\end{proof}
The following lemma is true in much greater generality (see \cite{Fu-Ka} 1.6.5 (3)) but 
is stated here only in the case we need.
\begin{lemma}\label{coefficientequation}
One has 
$$
	\cO_p\otimes_\Lambda^\bbL R\Gamma(\cO_K[1/p],\Lambda(\eta)(1))\isom 
	R\Gamma(\cO_K[1/p],\cO_p(\eta)(1)).
$$
In particular, one has a spectral sequence
$$
	\Tor_{r}^{\Lambda}\left( \cO_p, H^s(\cO_K[1/p],\Lambda(\eta)(1) ) \right)\\
\Rightarrow H^{s-r}(\cO_K[1/p],\cO_p(\eta)(1)).
$$
\end{lemma}
\begin{proof}(compare lemma \ref{perfect}).
This is shown for $\cO_K[1/pS]$  in \cite{Fu-Ka} 1.6.5 (3). 
With the distinguished triangle
$$
	\bigoplus_{v\in S}R\Gamma(\kappa(v),\Lambda(\eta)^{I_p})[-2]\to R\Gamma(\cO_K[1/p],\Lambda(\eta)(1))\to
	R\Gamma(\cO_K[1/pS],\Lambda(\eta)(1))
$$
the result follows in general. 
\end{proof}
Consider the triangle for cohomology with compact support
$$
	R\Gamma_c(\cO_K[1/p],\Lambda(\eta)(1))\to
	R\Gamma(\cO_K[1/p],\Lambda(\eta)(1))\to \bigoplus_{v\mid p}R\Gamma(K_v,\Lambda(\eta)(1))\oplus \Lambda(\eta)(1).
$$
For the computations of some Iwasawa modules, we need a local duality result: 
\begin{proposition}\label{Artin-Verdier} Let $T_p$ be a finitely generated projective $\Lambda_\cO$-module.
 Let $T_p^*:=\Hom_\cont(T_p,E_p/\cO_p)$.
If $T_p$ has a continuous $\Gal(\bar K_v/K_v)$-action one has an isomorphism
$$
	R\Gamma(K_v,T_p)\simeq R\Hom_{\cO_p}( R\Gamma(K_v,T_p^*(1)), E_p/\cO_p)[-2].
$$
\end{proposition}
\begin{proof}
This isomorphism is just a reformulation of 
the classical duality theorem (\cite{Milne} I 2.3.). 
\end{proof}
\begin{lemma}\label{lambdastructure} The modules
$$
H^2(K_v,\Lambda(\eta)(1))\mbox{ and }H^2(K_v,\Omega_\cO(1))
$$
for $v\mid p$
are finitely generated $\cO_p$-modules.
In particular, they are $\Lambda_\cO$-pseudo-null
 (resp. $\Omega_\cO$-pseudo-null).
\end{lemma}

\begin{proof}
The surjective map $\Omega_\cO\to \Lambda(\eta)$ induces a surjection
$$
	H^2(K_v,\Omega_\cO(1))\to H^2(K_v,\Lambda(\eta)(1))
$$
and it suffices to proof that  $H^2(K_v,\Omega_\cO(1))$ is a finitely generated $\cO_p$-module.
We have 
$$
	H^2(K_v,\Omega_\cO(1))\isom \prolim_n H^2(K_v\otimes_KK(p^n\frf),\cO_p(1))\simeq
	\cO_p[[\cG_\frf/D_v]]
$$
where $D_v$ is the decomposition group in $\cG_\frf$, because by local duality
$$
	H^2(K_v\otimes_K K(p^n\frf),\cO_p(1))\simeq H^0(K_v\otimes_KK(p^n\frf),E_p/\cO_p)^*\simeq \cO_p[\Gal(K(p^n\frf)/K)/D_{v,n}]
$$
where $D_{v,n}$ is the decomposition group of $v$ in $K(p^n\frf)$.
If $p$ is inert or ramified in $K$, the ramification group of $v$ has already finite index in $\cG_\frf$ and
the claim follows. If $p=\frp\frp'$ splits in $K$ and we decompose $K_\infty=K_\infty^{\frp}K_\infty^{\frp'}$,
with $K_\infty^{\frp}$ (resp. $K_\infty^{\frp'}$) the maximal unramified outside of $\frp$ 
(resp. outside of $\frp'$) subextension, then $\frp$ is totally ramified in $K_\infty^{\frp}$ and
finitely decomposed in $K_\infty^{\frp'}$ (see \cite{deShalit} II 1.9). Similarly for $\frp'$ and
it follows that in both cases $\cG_\frf/D_v$ is finite.
\end{proof}

We finally study the operation of twisting with a continuous character $\varrho:\Gamma\to \cO_p^*$. 
\begin{lemma}\label{twistinglemma}
	Let $\varrho:\Gamma\to \cO_p^*$ be a continuous character and consider $\Lambda(\varrho)$. Then
	there is an isomorphism of $\Gal(\bar{K}/K)$-modules depending on the generator
	$t_p(\varrho)$ of $\cO_p(\varrho)$
	$$
		\Lambda_\cO\simeq \Lambda(\varrho)
	$$
	given by $\gamma\mapsto \gamma\otimes \varrho(\gamma)t_p(\varrho)$. In particular,
	one has isomorphisms
	$$
		H^i(\cO_K[1/p],\Lambda(\eta))\simeq H^i(\cO_K[1/p],\Lambda(\eta\varrho))
	$$
	for all $i\ge 0$.
\end{lemma}
\begin{proof}
As $\Lambda_\cO\simeq \Lambda(\varrho)$ is obviously an isomorphism of $\Gal(\bar{K}/K)$-modules,
the statement follows.
\end{proof}

\section{Statement of the two main conjectures}

Recall the definition of the Iwasawa algebras $\Lambda$ and $\Omega$ from \ref{Iwasawa-alg-def}.
We will formulate in this section two main-conjectures. One for the ring $\Lambda$, which
corresponds to the statement of the \emph{main conjecture} decomposed into characters, and another for the ring $\Omega$, which is elsewhere called the \emph{equivariant main conjecture}. 

The $\Omega$-main-conjecture is apparently stronger because it is an equivariant statement, which
does not involve any characters. Nevertheless, we will
deduce the $\Omega$-main-conjecture from the $\Lambda$-main-conjecture for all $\cO_p(\eta)$
by a simple observation, which is inspired by the work of Burns-Greither \cite{Bu-Gr} for the cyclotomic case and was first explained by Witte in \cite{Witte}. 
\subsection{Preliminary notations}
Fix an integral ideal $\frf\subset \cO_K$ and define $\Omega$ and $\Lambda$ as in \ref{Iwasawa-alg-def}.
For each integral ideal $\fra$, which is prime to $6p\frf$, we denote by
$$
	\sigma_\fra\in \Omega_\cO\mbox{ resp. }\sigma_\fra\in\Lambda_\cO
$$
the Artin symbol $(\fra,K(p^\infty\frf)/K)$ resp. $(\fra,K_\infty/K)$ if no
confusion is possible. The elements $\N\fra-\sigma_\fra\in \Omega_\cO$ resp. in $\Lambda_\cO$
are no zero-divisors as their images under the canonical maps
$$
	\Omega_\cO\to \cO_p(\eta)\mbox{ resp. }\Lambda_\cO\to \cO_p(\eta)
$$
(mapping $\gamma\in \cG_\frf$ to $\eta(\gamma)$)
are given by $\N\fra-\eta^{-1}(\fra)$, which is invertible in $\cO_p(\eta)\otimes_{\cO_p}E$
for all non-trivial characters $\eta$  of finite order. Denote by
$$
	\cJ_\Omega\mbox{ resp. }\cJ_\Lambda
$$
the annihilator of $\mu_{p^\infty}(K(p^\infty\frf))$ resp. $\mu_{p^\infty}(K_\infty)$
in $\Omega$ resp. $\Lambda$. It is easy to see that $\cJ_\Omega$ resp. $\cJ_\Lambda$ is generated by $\N\fra-\sigma_\fra$
for all $\fra$ prime to $6p\frf$. 
Note that $\cJ_\Omega$ (resp. $\cJ_\Lambda$) can also be described as the kernel
of the map $\Omega\to \bbZ_p$ (resp. $\Lambda\to \bbZ_p$) induced by the cyclotomic
character (resp. the restriction of the cyclotomic character to $\Gamma$).
This implies that $\Omega/\cJ_\Omega$ and $\Lambda/\cJ_\Lambda$
are isomorphic to $\bbZ_p$ and are hence pseudo-null modules. In particular, for all
prime ideals $\frq\subset \Omega$ resp. $\frq\subset \Lambda$ of height $1$,
$$
	(\cJ_\Omega)_\frq\simeq(\Omega)_\frq\mbox{ resp. }(\cJ_\Lambda)_\frq\simeq(\Lambda)_\frq.
$$
\begin{lemma}
The ideals $\cJ_\Omega\subset \Omega$ resp. $\cJ_\Lambda\subset \Lambda$ are perfect
$\Omega$- resp. $\Lambda$-modules.
\end{lemma}
\begin{proof}
	As $\Lambda$ is regular, the quotient $\Lambda/\cJ_\Lambda$ is perfect and the
statement is clear in this case. For $\Omega\isom \bbZ_p[\Delta][[T,S]]$, write 
$\Delta=\Delta'\times \Delta''$, where the order of $\Delta'$ is divisible by $p$ and
the order of $\Delta''$ is prime to $p$. Then the cyclotmic character $\Omega\to \bbZ_p$
factors through the regular ring $A:=\bbZ_p[\Delta''][[T,S]]$. Note that $A$ is a quotient 
of $\Omega$ but also a direct summand. In particular, $A$ is projective as $\Omega$-module.
Let $A\to \bbZ_p$ be given by the cyclotomic character. As $A$ is regular, this has a finite
resolution by free $A$-modules of finite rank, so in particular by projective $\Omega$-modules.
This implies that $\cJ_\Omega$ as the kernel of $\Omega\to \bbZ_p$ is also perfect.
\end{proof}

\subsection{The $\Lambda$-main-conjecture}

Consider a character $\chi:G(\frf_\chi)\to E^*$ of conductor $\frf_\chi$ and
fix $\fra$ prime to $6p\frf_\chi$.
In definition \ref{twisted-euler-defn} we have defined elements 
$$
_\fra\zeta_{K_n}(\chi)\in H^1(\cO_{K_n}[1/p],\cO_p(\chi)(1)),
$$
which are part of an Euler system in the sense of definition \ref{rubin-defn}. Note that we retain the subscript $\fra$ for this element for the remainder of the paper.  In particular, 
these elements are norm compatible in the
$K_\infty$-direction and we can define
$$
	{_\fra\zeta(\chi)}:=\prolim_n{_\fra\zeta_{K_n}(\chi)}\in H^1(\cO_{K}[1/p],\Lambda(\chi)(1)).
$$
We define
\begin{equation}\label{zetachidefn}
	\zeta(\chi):=(\N\fra-\sigma_\fra)^{-1}{_\fra\zeta(\chi)}\in H^1(\cO_{K}[1/p],\Lambda(\chi)(1))\otimes_{\Lambda_\cO}Q(\Lambda_\cO),
\end{equation}
where $Q(\Lambda_\cO)$ is the total quotient ring of $\Lambda_\cO$. Using proposition \ref{unitproperties} (3) this
element is independent of $\fra$ and by definition it satisfies
$$
	(\N\fra-\sigma_\fra)\zeta(\chi)\in H^1(\cO_{K}[1/p],\Lambda(\chi)(1)).
$$
We consider the submodule 
$\cJ_\Lambda(\zeta(\chi))\subset H^1(\cO_{K}[1/p],\Lambda(\chi)(1))$. Recall that $\zeta(\chi)$ depends
on our choice of a generator $t_p(\chi)$ of the lattice $\cO_p(\chi)$.
By definition we have 
$$
	\Lambda_\cO(_\fra\zeta(\chi))\subset \cJ_\Lambda(\zeta(\chi))\subset H^1(\cO_{K}[1/p],\Lambda(\chi)(1))
$$
and $\cJ_\Lambda(\zeta(\chi))$ is generated by the  $_\fra\zeta(\chi)$ for all $\fra$ prime to $6p\frf_\chi$.
This inclusion induces a morphism of perfect complexes
$$
	\kappa_\chi:\cJ_\Lambda(\zeta(\chi))\to R\Gamma(\cO_{K}[1/p],\Lambda(\chi)(1))[1].
$$
\begin{theorem}[Main-Conjecture]\label{lambda-mc} For each character $\chi:G(\frf_\chi)\to E^*$ of conductor $\frf_\chi$ 
\begin{itemize}
\item[1)] $H^0 (\cO_{K}[1/p],\Lambda(\chi)(1))=0$.
\item[2)] $H^1 (\cO_{K}[1/p],\Lambda(\chi)(1))$ has $\Lambda_\cO$ rank $1$ and 
$$
	H^1 (\cO_{K}[1/p],\Lambda(\chi)(1))/\cJ_\Lambda(\zeta(\chi))
$$
is torsion.
\item[3)] $H^2(\cO_{K}[1/p],\Lambda(\chi)(1))$ is a $\Lambda_\cO$-torsion module.
\end{itemize}
The map $\kappa_\chi$ is good in the sense of Knudsen and Mumford (see \ref{determinants}) and
has divisor $\Div(\kappa_\chi)=0$. In particular, $\kappa_\chi$ induces 
an isomorphism of $\Lambda_\cO$-modules
$$
\kappa_\chi:
\Det_{\Lambda_\cO} \left( H^1(\cO_{K}[1/p],\Lambda(\chi)(1))/\cJ_\Lambda(\zeta(\chi))\right)\simeq \Det_{\Lambda_\cO} H^2(\cO_{K}[1/p],\Lambda(\chi)(1)).
$$
\end{theorem}
This theorem will be proved in section \ref{proof-of-lambda-mc}. Note that the statement
is for all primes $p$ with no exceptions.
\begin{corollary}\label{localizationcor} Suppose that $\frf_\chi\mid \frf$ and
consider the image of $\zeta(\chi)$ in $H^1(\cO_{K}[1/p\frf],\Lambda(\chi)(1))$ and let 
$$
	\zeta(\chi,\frf):=\prod_{\frl\mid\frf, \frl\nmid p}(1-\chi(\frl)\Frob_\frl^{-1})\zeta(\chi).
$$
Then the inclusion
$$
	\wt\kappa_\chi:\cJ_\Lambda(\zeta(\chi,\frf))\to R\Gamma(\cO_{K}[1/p\frf],\Lambda(\chi)(1))[1].
$$
is good in the sense of \ref{determinants} and induces an isomorphism of $\Lambda_\cO$-modules 
$$
\wt\kappa_\chi:
\Det_{\Lambda_\cO} \left( H^1(\cO_{K}[1/p\frf],\Lambda(\chi)(1))/\cJ_\Lambda(\zeta(\chi,\frf))\right)\simeq \Det_{\Lambda_\cO} H^2(\cO_{K}[1/p\frf],\Lambda(\chi)(1)).
$$
\end{corollary}
\begin{proof}
This follows from the theorem, the localization sequence
$$
	\bigoplus_{\frl\mid\frf,\frl\nmid p}R\Gamma_{\kappa(\frl)}(\cO_{K_\frl},\Lambda( \chi)(1))\to
	R\Gamma(\cO_K[1/p],\Lambda( \chi)(1))\to R\Gamma(\cO_K[1/p\frf],\Lambda( \chi)(1))
$$	
and the fact that by lemma \ref{perfectness}
$$
\Det_{\Lambda_\cO}R\Gamma_{\kappa(\frl)}(\cO_{K_\frl},\Lambda( \chi)(1))=\Lambda_\cO(1-\chi(\frl)\Frob_\frl^{-1})^{-1}.
$$
\end{proof}
\begin{remark}
Observe that
our formulation here follows \cite{HK1} and
is different from the classical approach by Rubin. Rubin decomposes the Iwasawa modules into $\chi$-eigenspaces, we use instead cohomology with coefficients in $\cO_p(\chi)$. This approach
avoids many problems with the $\chi$-eigenspaces and is very close to the spirit
of the Tamagawa number conjecture.
\end{remark}
%

\subsection{The $\Omega$-main-conjecture}

We are ultimately interested in an equivariant version of the  $\Lambda$-main-conjecture.
Recall that $\Omega=\Lambda(\cG_\frf)$. We admit the following hypothesis. 
\begin{conjecture}\label{mu=0} 
Let $\frq$ be a height one prime ideal of $\Omega$ containing $p$, then
$$
	H^2(\cO_K[1/p\frf],\Omega(1))_\frq=0.
$$
\end{conjecture}
This conjecture is essentially equivalent to the vanishing of the $\mu$-invariant
for the maximal abelian $\bbZ_p$-extension of $K_\infty$.
Using results of Gillard, we show in \ref{muverification} that this conjecture holds for 
primes $p\nmid 6$, which
are split in $K$: 
\begin{theorem}[see corollary \ref{muverification}]
In the case that $p\nmid 6$ splits in $K/\bbQ$, Conjecture \ref{mu=0} is true.
\end{theorem}
Recall the Euler system of elliptic units presented in section \ref{elliptic-units}. Consider 
for $\fra$ prime to $6p\frf$
$$
_\fra\zeta(\frf):=\varprojlim_n({_\fra}\zeta_{\frf p^n})\in H^1(\cO_{K}[1/p\frf],\Omega(1)).
$$
Define as in (\ref{zetachidefn})
$$
	\zeta(\frf):=(\N\fra-\sigma_\fra)^{-1}{_\fra\zeta(\frf)}\in 
	H^1(\cO_{K}[1/p\frf],\Omega(1))\otimes_\Omega Q(\Omega),
$$
where $Q(\Omega)$ is the total quotient ring of $\Omega$. Note that $\zeta(\frf)$ 
is again independent of $\fra$. We have by definition
$$
	\cJ_\Omega(\zeta(\frf))\subset H^1(\cO_{K}[1/p\frf],\Omega(1))
$$
and $\cJ_\Omega(\zeta(\frf))$ is generated by the $_\fra\zeta(\frf)$. Consider the
inclusion of perfect complexes
$$
	\kappa_\frf:\cJ_\Omega(\zeta(\frf))\to R\Gamma(\cO_K[1/p\frf],\Omega(1))[1].
$$
\begin{theorem}[Equivariant Main Conjecture]\label{omega-mc} Fix an non-zero ideal $\frf\subset \cO_K$.
Then 
\begin{itemize}
\item[1)] $H^0(\cO_{K}[1/p\frf],\Omega(1))=0$
\item[2)] $H^1(\cO_{K}[1/p\frf],\Omega(1))$ has $\Omega$ rank $1$ and 
$$
	H^1(\cO_{K}[1/p\frf],\Omega(1))/\cJ_\Omega(\zeta(\frf))
$$
is torsion.
\item[3)] $H^2(\cO_{K}[1/p\frf],\Omega(1))$ is an
$\Omega$-torsion module.
\end{itemize}
Assume conjecture \ref{mu=0}, i.e., $H^2(\cO_K[1/p\frf],\Omega(1))_\frq=0$ for all height one prime ideals
with $p\in \frq$, then
the map $\kappa_\frf$ is good in the sense of \ref{determinants} and has divisor $\Div(\kappa_\frf)=0$.
In particular, $\kappa_\frf$ induces an isomorphism of $\Omega$-modules
$$
	\Det_\Omega\left( H^1(\cO_{K}[1/p\frf],\Omega(1))/\cJ_\Omega(\zeta(\frf))\right)\simeq 
	\Det_\Omega\left( H^2(\cO_{K}[1/p\frf],\Omega(1))\right).
$$
\end{theorem}
Note that conjecture \ref{mu=0} holds by corollary \ref{muverification} for all prime numbers $p\nmid 6$, which
split in $K$.

This theorem will be proved in section \ref{proof-of-omega-mc}.

\subsection{Relation to the classical Iwasawa main conjecture}
Recall that $\Omega\simeq \bbZ_p[\Delta][[S,T]]$ and let us assume
that the order of $\Delta$ is prime to $p$. This implies that
$\Omega_\cO$ decomposes
$$
	\Omega_\cO\simeq \prod_{\chi\in\wh\Delta}\Lambda(\chi).
$$
Moreover, this decomposition is given by the projectors $p_{\chi^{-1}}$, which are
in this case are already defined over $\cO_p$ 
(note that the image of $\delta_e\in \bbZ_p[\Delta]$ under the projector $p_{\chi^{-1}}$
is $t_p(\chi)$). In particular, $\Lambda(\chi)=\Lambda\otimes_{\bbZ_p}\cO_p(\chi)$ and $\Omega_\cO$ is a
product of regular local rings.
Note also that by \cite{Witte} 3.6. i) in this case the complex $R\Gamma(\cO_{K}[1/p],\Omega_\cO(1))$
(without $\frf$) is also perfect. 

Using the above decomposition we get
$$
	p_{\chi^{-1}}H^1(\cO_{K}[1/p],\Omega_\cO(1))\simeq H^1(\cO_{K}[1/p],\Lambda(\chi)(1)).
$$
We define as usual 
\begin{align*}
	\cA_\infty&:=\prolim_n \Pic(\cO_{K(p^n\frf)})\otimes_\bbZ\bbZ_p\\
	\cE_\infty&:=\prolim_n \cO_{K(p^n\frf)}^*\otimes_\bbZ\bbZ_p\\
	\cE_\infty'&:=\prolim_n \cO_{K(p^n\frf)}[1/p]^*\otimes_\bbZ\bbZ_p.
\end{align*}
Note that $\cJ_\Omega(\zeta(\frf))\subset \cE_\infty'$. If $\cI_\Omega\subset \Omega$
denotes the augmentation ideal we have in fact $\cI_\Omega\cJ_\Omega(\zeta(\frf))\subset \cE_\infty$.
The quotient $\cJ_\Omega(\zeta(\frf))/\cI_\Omega\cJ_\Omega(\zeta(\frf))$ is pseudo-null.
The elliptic units $\cC_\infty\subset \cE_\infty$ defined by Rubin in \cite{Rubin}
are (up to torsion) the $\Omega$-submodule generated by $\cI_\Omega\cJ_\Omega(\zeta(\frg))$  for
all integral ideals $\frg\mid\frf$.

The relation of the above Iwasawa modules to the ones used in this paper is given as follows:
\begin{lemma}[\cite{Bu-Gr} 5.1., \cite{Witte} 4.1 and 4.2]\label{classicalcomparison}
One has an isomorphism
$$
	\cE_\infty'\simeq H^1(\cO_{K}[1/p],\Omega(1))
$$
and an exact sequence
\begin{align*}
0&\to \cE_\infty\to \cE_\infty'\to \bigoplus_{v\mid p}\cO_p[[\cG_\frf/D_v]]\to \cA_\infty\to\\ &\to H^2(\cO_{K}[1/p],\Omega(1))\to \bigoplus_{v\mid p}\cO_p[[\cG_\frf/D_v]]\to \cO_p\to 0,
\end{align*}
where $D_v$ is the decomposition group at $v$.
\end{lemma}
Note that $\bigoplus_{v\mid p}\cO_p[[\cG_\frf/D_v]]$ is a finitely generated $\cO_p$-module
(compare lemma \ref{lambdastructure}) 
as there are only finitely many primes above $p$ in $K(p^\infty\frf)$. In particular,
it is a pseudo-null $\Omega_\cO$-module. Consider now in each $H^1(\cO_{K}[1/p],\Lambda(\chi)(1))$
the sub-module
$$
	(\cI_\Lambda\cJ_\Lambda(\zeta(\chi)))\subset H^1(\cO_{K}[1/p],\Lambda(\chi)(1))\simeq p_{\chi^{-1}}{\cE_\infty'}
$$
(more traditionally one writes ${\cE_\infty'}^{\chi^{-1}}=p_{\chi^{-1}}{\cE_\infty'}$ etc.).
Here $\zeta(\chi)$ is as in theorem  \ref{lambda-mc}.
Then theorem \ref{lambda-mc} gives a canonical isomorphism
$$
	\Det_{\Lambda_\cO}\left(p_{\chi^{-1}}{\cE_\infty}/\cI_\Lambda\cJ_\Lambda(\zeta(\chi))\right)\simeq
	\Det_{\Lambda_\cO}\left(p_{\chi^{-1}}\cA_\infty\right),
$$
which holds with no restriction on $p$ besides the fact that $p$ is prime to the order of $\Delta$.

\subsection{Conjecture \ref{mu=0} and the vanishing of the $\mu$-invariant}
In this section we show that the results of Gillard \cite{Gillard}
imply the conjecture \ref{mu=0} for $p\nmid 6$, which are split
in $K$.

Assume that $p=\frp\frp'$ in $K$ and 
let $K\subset K_\infty^\frp\subset K_\infty$ (resp. $K_\infty^{\frp'}$)
be the $\bbZ_p$-extension of $K$, which is
unramified outside of $\frp$ (resp. $\frp'$). Recall from (\ref{Gfdefn}) that we fixed a splitting
$\cG_\frf\simeq \Delta\times \Gamma$ and define $L/K$ such that $\Gal(L/K)\simeq \Delta$.
Let $F_\infty:=LK_\infty^\frp$ be the compositum, then 
$\Gal(F_\infty/K)\isom \Delta\times \Gal(K_\infty^\frp/K)$
and 
$$
	\cG_\frf\isom \Gal(K_\infty^{\frp'}/K)\times \Gal(F_\infty/K).
$$
Define $\cH:=\Gal(K(\frf p^\infty)/F_\infty)\isom\Gal(K_\infty^{\frp'}/K)\isom \bbZ_p$, so that 
\begin{equation}\label{Hdefn}
	0\to \cH\to \cG_\frf\to \Gal(F_\infty/K)\to 0
\end{equation}
is exact.

Let $M_\infty$ be the maximal abelian
$\bbZ_p$-extension of $F_\infty$, which is unramified outside of $\frp$. 
Gillard proves:
\begin{theorem}[Gillard \cite{Gillard} 3.4.] Let $p\nmid 6$ be split in $K$. The group 
$\Gal(M_\infty/F_\infty)$ has no $\bbZ_p$-torsion. In particular, it is a finitely generated
$\bbZ_p$-module.
\end{theorem}
We want to apply this theorem to prove conjecture \ref{mu=0}, i.e., we want to show that $H^2(\cO_{K}[1/p\frf],\Omega(1))_\frq=0$
for all height one prime ideals with $p\in \frq$. Note first that we can
make the flat base change to $\cO_p$.
We study $ H^2(\cO_{K}[1/p\frf],\Lambda_\cO(\Gal(F_\infty/K))(1))$, where 
$\Lambda_\cO(\Gal(F_\infty/K))$ is the Iwasawa algebra of $\Gal(F_\infty/K)$.

Let
$$
	\cA(F_\infty):=\prolim_n (\Pic(\cO_{F_n})\otimes_\bbZ\cO_p)
$$
be the inverse limits of the class groups of the fields $F_n:=K_nL$ so that
$F_\infty= \bigcup_n F_n$. Then $\cA(F_\infty)$ is a $\cO_p$-module, which is a quotient of
$\Gal(M_\infty/F_\infty)\otimes_{\bbZ_p}\cO_p$. 
The above theorem implies that $\cA(F_\infty)$
is a finitely generated $\cO_p$-module. 

\begin{corollary}
With the above notations
$$
	 H^2(\cO_{K}[1/p\frf],\Lambda_\cO(\Gal(F_\infty/K))(1))
$$
is a finitely generated $\cO_p$-module. In particular, $H^2(\cO_{K}[1/p\frf],\Omega_\cO(1))$ is
a finitely generated $\Lambda_\cO(\cH)$-module.
\end{corollary}
\begin{proof}
As in lemma \ref{classicalcomparison} one has an exact sequence
$$
	\cA(F_\infty)\to H^2(\cO_{K}[1/p\frf],\Lambda_\cO(\Gal(F_\infty/K))(1))\to 
	\bigoplus_{v\mid p\frf}
	\cO_p[[\Gal(F_\infty/K)/D_v]]
$$
where $D_v$ is the decomposition group. As $p$ splits in $K$, all primes not
above $\frp$ are finitely decomposed in $F_\infty$ (see \cite{deShalit} II 1.9) and $\frp$ is
completely ramified in $K^{\frp}_\infty$. It follows that the $H^2(K_v,\Lambda_\cO(\Gal(F_\infty/K))(1))= \cO_p[[\Gal(F_\infty/K)/D_v]]
$
are finitely generated $\cO_p$-modules.
We consider now $ H^2(\cO_{K}[1/p\frf],\Omega_\cO(1))$ as a compact $\Lambda_\cO(\cH)$-module. Using \cite{Fu-Ka} 1.6.5 (3) one sees that 
$$
	H^2(\cO_{K}[1/p\frf],\Omega_\cO(1))\otimes_{\Lambda_\cO(\cH)}\cO_p\simeq
	H^2(\cO_{K}[1/p\frf],\Lambda_\cO(\Gal(F_\infty/K))(1)).
$$
It follows from Nakayama's lemma (see \cite{NSW} 5.2.18)  that 
$H^2(\cO_{K}[1/p\frf],\Omega_\cO(1))$ is a finitely generated $\Lambda_\cO(\cH)$-module.
\end{proof}
We conclude with the following general structure result.
\begin{lemma}\label{vanishingforheightoneprimes}
Let $M$ be an $\Omega_\cO$-module which is finitely generated as $\Lambda_\cO(\cH)$-module.
Then for any height one prime
ideal $\frq\subset \Omega_\cO$ with $p\in \frq$, one has
$$
	M_\frq=0.
$$
\end{lemma}
\begin{proof}
Let $\wt M:=M/\frq M$,
$\wt\Omega:=\Omega_\cO/\frq\Omega_\cO$. We denote by $\kappa(\frq)$
the residue class field of $\frq$.
By Nakayama's lemma it suffices to show that
$$
	M_\frq/\frq M_\frq=\wt M\otimes_{\wt \Omega}\kappa(\frq)=0,
$$
By the exact sequence (\ref{Hdefn}), we have 
$\Omega_\cO\otimes_{\Lambda_\cO(\cH)}\cO_p\simeq \Lambda_\cO(\Gal(F_\infty/K))$
and we let 
$$
	I:=\ker(\Omega_\cO\to \Lambda_\cO(\Gal(F_\infty/K))).
$$
 By our assumption, $M/IM$ is a finitely
generated $\cO_p$-module.
Identify $\Lambda_\cO(\Gal(F_\infty/K))\simeq \cO_p[\Delta][[u]]$ and
choose $\wt u\in \Omega_\cO$ mapping to $u\in \cO_p[\Delta][[u]]$.
Note that $\wt u\notin \frq$ as otherwise $p,\wt u\in \frq$ and
$\frq$ could not have height $1$.
We show that 
$$
	\wt M\otimes_{\wt\Omega}\wt\Omega[\wt u^{-1}]=0,
$$
which gives the desired result, as $\wt\Omega[\wt u^{-1}]\subset \kappa(\frq)$.
Let $\wt I:=I/\frq I\subset \wt \Omega$.
As $p\in \frq$  the $\wt\Omega$-module $\wt M/\wt I\wt M$ is
finitely generated $\cO_p/p\cO_p$-module, hence a finite group. 
This implies that there is an integer $k$ such that 
$\wt u^k(\wt M/\wt I\wt M)= \wt u^{k+1}(\wt M/\wt I\wt M)$.
As $\wt u $ is in the radical of $\wt\Omega$, Nakayama's lemma shows that $\wt u^k(\wt M/\wt I\wt M)=0$.
This shows $(\wt M/\wt I\wt M)\otimes_{\wt\Omega/\wt I\wt\Omega}\wt\Omega/\wt I\wt\Omega[\wt u^{-1}]=0$.
As $\wt I$ is in the radical of $\wt\Omega[\wt u^{-1}]$ Nakayama's lemma implies that
$\wt M\otimes_{\wt\Omega}\wt\Omega[\wt u^{-1}]=0$.
\end{proof}

\begin{corollary}[Conjecture \ref{mu=0} for split primes]\label{muverification}
Let $p$ be a prime, which splits in $K$ and assume that 
$p\nmid 6$. Then,
for any height one prime
ideal $\frq\subset \Omega$ with $p\in \frq$, one has
$$
	H^2(\cO[1/p\frf],\Omega(1))_\frq=0.
$$
\end{corollary}

\section{Proof of the $\Lambda$-main-conjecture}\label{proof-of-lambda-mc}

In this section we prove the $\Lambda$-main-conjecture  as formulated
in theorem \ref{lambda-mc}.

\subsection{Reduction to characters of big enough level}

Let $\chi:G(\frf_\chi)\to E^*$ be a character of conductor $\frf_\chi$ and let $\fra$ be an
integral ideal prime to $6p\frf_\chi$.
Consider the submodule $\Lambda_\cO({_\fra}\zeta(\chi))\subset H^1(\cO_K[1/p],\Lambda(\chi)(1))$.
\begin{lemma}\label{charactertwisting} Consider a continuous character $\varrho:\Gamma\to \cO_p^*$. Then
the twisting map of lemma \ref{twistinglemma}
maps $_\fra\zeta(\chi)$ to
$$
	_\fra\zeta(\chi\varrho)\in H^1(\cO_K[1/p],\Lambda(\chi\varrho)(1)).
$$
In particular, $\cJ\zeta(\chi)$ is mapped to 
$\cJ\zeta(\chi\varrho)$ and the $\Lambda$-main-conjecture is compatible with twists.
\end{lemma}
\begin{proof} As the twisting morphism maps the generator $t_p(\chi)$ to $t_p(\chi\rho)$,
this is a direct consequence of the construction of $_\fra\zeta(\chi)$ in \ref{twisted-euler-defn}
(see also \cite{Ho-Ki} section 1.2). As $\Lambda(\chi)\isom \Lambda(\chi\varrho)$ as
$\Gal(\bar{K}/K)$-modules, it is clear that
$$
	R\Gamma(\cO_K[1/p],\Lambda(\chi)(1))\isom R\Gamma(\cO_K[1/p],\Lambda(\chi\varrho)(1)).
$$
\end{proof}

This lemma allows us to reduce the $\Lambda$-main-conjecture for $\chi$ to
the one for $\eta:=\chi\varrho$ using the isomorphisms in \ref{twistinglemma}.
Choose $\varrho$ such that the level of $\eta=\chi\varrho$ is big enough.
This gives:
\begin{corollary}\label{reductioncor}
To prove the $\Lambda$-main-conjecture, it suffices to consider characters $\eta$ of 
level big enough.
\end{corollary}

\subsection{Divisibility obtained from the Euler system}
In this section we use the Euler system defined by the elliptic units to
prove one divisibility in the statement of the $\Lambda$-main-conjecture.
We consider characters $\eta$ of level big enough.

Let us define a subgroup of $H^2(\cO_K[1/p],\Lambda(\eta)(1))$, which plays the role of
the Selmer group.
\begin{definition}
Let 
$$
H^2_0(\cO_K[1/p],\Lambda(\eta)(1)):=
\ker\left(H^2(\cO_K[1/p],\Lambda(\eta)(1))\to 
\bigoplus_{v\mid p}H^2(K_v,\Lambda(\eta)(1))\right).
$$
\end{definition}
The theory of Euler systems 
gives:
\begin{theorem}\label{divisibility} Let $\eta$ be a character of conductor $\frf_\eta$ and level $n$,
chosen so big that $\cO_p(\eta)$ is ramified at all places $v\mid p$, then:
	\begin{itemize}
	\item[1)] $H^2(\cO_K[1/p],\Lambda(\eta)(1))$ is $\Lambda_\cO$-torsion.
	\item[2)] $H^1(\cO_K[1/p],\Lambda(\eta)(1))$ has $\Lambda_\cO$-rank one.
	\item[3)] $H^1(\cO_K[1/p],\Lambda(\eta)(1))/\Lambda_\cO{_\fra}\zeta(\eta)$ is  $\Lambda_\cO$-torsion.
	\item[4)] Identify the $\Lambda_\cO$-determinants of the torsion modules 
	$$
	H^2_0(\cO_K[1/p],\Lambda(\eta)(1)) \mbox{ and }H^1(\cO_K[1/p],
	\Lambda(\eta)(1))/\Lambda_\cO{_\fra}\zeta(\eta)
	$$
	with invertible submodules of the total quotient ring $Q(\Lambda_\cO)$. Then:
	$$
		\Det_{\Lambda_\cO} \left(H^2_0(\cO_K[1/p],\Lambda(\eta)(1))\right)\subset\Det_{\Lambda_\cO} \left(H^1(\cO_K[1/p],\Lambda(\eta)(1))/\Lambda_\cO{_\fra}\zeta(\eta)\right).
	$$
	\end{itemize}
\end{theorem}
\begin{proof}
	This is a consequence of the theory of Euler systems. We follow the
	exposition in Rubin, as this is closest to our setting. Let us begin by checking the
	hypothesis $\mbox{Hyp}(K_\infty/K)$ and $\mbox{Hyp}(K_\infty,\cO_p(\eta)(1))$
	in Rubin \cite{Rubin2} 2.3.3. This is clear for $\mbox{Hyp}(K_\infty/K)$ 
	as $K$ is imaginary quadratic. For $\mbox{Hyp}(K_\infty,\cO_p(\eta)(1))$ we take $\tau=\id$. 
	We also remark that it is
	clear from the definition that $H^2_0(\cO_K[1/p],\Lambda(\eta)(1))=X_\infty$ 
	in Rubin's notation. 
	As our element ${_\fra\zeta(\eta)}$ is non-torsion by \ref{zetaexplicit}, the
	Theorem 2.3.2 in \cite{Rubin2} implies that $H^2_0(\cO_K[1/p],\Lambda(\eta)(1))$ is 
	$\Lambda_\cO$-torsion. As $H^2(K_v,\Lambda(\eta)(1))$ for $v\mid p$ is $\Lambda_\cO$-torsion by 
	lemma \ref{lambdastructure}, it follows that $H^2(\cO_K[1/p],\Lambda(\eta)(1))$ is 
	$\Lambda_\cO$-torsion as well, which shows 1). To show 2) note that 
	$H^0(\cO_K[1/p],\Lambda(\eta)(1))=0$ by (\ref{h0vanishing}). Then, 2) follows from
	the formula
	\begin{multline*}
	\rk_{\Lambda_\cO} H^1(\cO_K[1/p],\Lambda(\eta)(1))-\rk_{\Lambda_\cO} H^2(\cO_K[1/p],\Lambda(\eta)(1))=\\
	\rk_{\cO_p}H^2(\cO_K[1/p], \cO_p(\eta)(1))-\rk_{\cO_p}H^1(\cO_K[1/p],\cO_p(\eta)(1)).
	\end{multline*}
	(see \cite{Ts} prop. 9.2. (3) for example) and
	$$
	\rk_{\cO_p}H^2(\cO_K[1/p],\cO_p(\eta)(1))-\rk_{\cO_p}
	H^1(\cO_K[1/p],\cO_p(\eta)(1))=\rk_{\cO_p}\cO_p(\eta)=1
	$$
	(see \cite{Ja} Lemma 2). As $\Lambda_\cO({_\fra\zeta(\eta)})\subset  H^1(\cO_K[1/p],\Lambda(\eta)(1))$
	is a non-torsion submodule by \ref{zetaexplicit}, we get 3). 
	For the statement in 4) we have to consider 
	$$
	\ind_{\Lambda_\cO}({_\fra\zeta(\eta)}):=
	\{\phi({_\fra\zeta(\eta)})|\phi\in\Hom_{\Lambda_\cO}(H^1(\cO_K[1/p],\Lambda(\eta)(1)),\Lambda_\cO)
	$$
	as
	defined by Rubin in \cite{Rubin2} p. 41. Our $H^1(\cO_K[1/p],\Lambda(\eta)(1))$
	is Rubin's $H^1_\infty(K,\cO_p(\eta)(1))$ by \cite{Rubin2} Corollary B.3.5.
	By the structure theory of $\Lambda_\cO$-modules, we can find a pseudo-isomorphism
	$$
		H^1(\cO_K[1/p],\Lambda(\eta)(1))\to \Lambda_\cO\oplus H^1(\cO_K[1/p],\Lambda(\eta)(1))_\tors.
	$$
	Let $\phi:H^1(\cO_K[1/p],\Lambda(\eta)(1))\to \Lambda_\cO$ be 
	the projection onto $\Lambda_\cO$, then the kernel of $\phi$ is
	torsion and one has an exact sequence of $\Lambda_\cO$-torsion modules
	$$
	0\to \ker\phi\to H^1(\cO_K[1/p],\Lambda(\eta)(1))/\Lambda_\cO{_\fra\zeta(\eta)}\to \Lambda_\cO/\Lambda_\cO\phi({_\fra\zeta(\eta)})\to 0.
	$$
	This gives inside $Q(\Lambda_\cO)$, using $\Det_{\Lambda_\cO}^{-1}\ker\phi\subset \Lambda_\cO$,
	$$
	\Det_{\Lambda_\cO}^{-1} \left(H^1(\cO_K[1/p],\Lambda(\eta)(1))/\Lambda_{\cO}{_\fra\zeta(\eta)}\right)\subset \phi({_\fra\zeta(\eta)})\Lambda_\cO\subset \ind_{\Lambda_\cO}({_\fra\zeta(\eta)}).
	$$
	Finally, theorem 2.3.3 in \cite{Rubin2} shows 
	$$
		\ind_\Lambda({_\fra\zeta(\eta)})\subset \Det_{\Lambda_\cO}^{-1} H^2_0(\cO_K[1/p],\Lambda(\eta)(1)),
	$$
	which gives statement 4).
\end{proof}
Next, we strengthen the divisibility of theorem \ref{divisibility}. For this, we need a lemma:

\begin{corollary}\label{divcorollary} Let $\eta$ be as in theorem \ref{divisibility}. Under the isomorphism
	$$
		Q(\Lambda_\cO){_\fra\zeta(\eta)}\simeq H^1(\cO_K[1/p],\Lambda(\eta)(1))\otimes_{\Lambda_\cO}Q(\Lambda_\cO)
	$$
	one has an inclusion of $\Lambda_\cO$-modules
	$$
		\Det_{\Lambda_\cO} \left(H^2(\cO_K[1/p],\Lambda(\eta)(1))\right)\subset\Det_{\Lambda_\cO} \left(H^1(\cO_K[1/p],\Lambda(\eta)(1))/\Lambda_\cO{_\fra\zeta(\eta)}\right).
	$$
\end{corollary}

\begin{proof}
	By definition of $H^2_0(\cO_K[1/p],\Lambda(\eta)(1))$ 
	we have an exact sequence
	\begin{align*}
	0\to H^2_0(\cO_K[1/p],\Lambda(\eta)(1))\to H^2(\cO_K[1/p],\Lambda(\eta)(1))\to\\
	\to \bigoplus_{v\mid p}H^2(K_v,\Lambda(\eta)(1))
	\end{align*}
	By lemma \ref{lambdastructure} the modules $H^2(K_v,\Lambda(\eta)(1))$ are finitely generated
	$\cO_p$-modules and hence pseudo-null.
	It follows that  inside $Q(\Lambda_\cO)$
	$$
	\Det_{\Lambda_\cO}^{-1} H^2_0(\cO_K[1/p],\Lambda(\eta)(1))= \Det_{\Lambda_\cO}^{-1} H^2(\cO_K[1/p],\Lambda(\eta)(1)).
	$$
	This, together with the vanishing of 
	$H^0(\cO_K[1/p],\Lambda(\eta)(1))$ by (\ref{h0vanishing}) and the divisibility in theorem \ref{divisibility} gives the result.
\end{proof}

\subsection{Reduction to the Tamagawa number conjecture }
In this section we reduce the $\Lambda$-main-conjecture \ref{lambda-mc}
to the Tamagawa number conjecture.
In this section $\eta$ is a character of conductor $\frf_\eta$ and level $n$, chosen so that
$\cO_p(\eta)$ is ramified at all  $v\mid p$.

Observe that $\Lambda_\cO$ is a product of regular local noetherian rings, so that
we can use the functor $\Div$ from \ref{determinants}. Consider the inclusion of perfect complexes
\begin{equation}\label{kappadefn}
	\kappa_\eta:\cJ_\Lambda(\zeta(\eta))\to R\Gamma (\cO_K[1/p],\Lambda(\eta)(1))[1].
\end{equation}
To compare this with the
divisibility results obtained from the theory of Euler systems, we also consider
the inclusion
$$
	\tau_\eta:\Lambda_\cO(_\fra\zeta(\eta))\subset \cJ_\Lambda(\zeta(\eta)).
$$
By theorem \ref{divisibility},
both $\kappa_\eta$ and $\tau_\eta$ are isomorphisms after tensoring with 
$Q(\Lambda_\cO)$, hence $\kappa_\eta$ and $\tau_\eta$ are good as defined
in \ref{determinants} and we can consider $\Div(\kappa_\eta)$ and $\Div(\tau_\eta)$ on $\Spec\Lambda_\cO$. 
Applying (\ref{divequation}) to $\kappa_\eta$ one gets
$$
	\Det_{\Lambda_\cO}(\cJ_\Lambda(\zeta(\eta)))(\Div(\kappa_\eta))=\Det_{\Lambda_\cO}\left( R\Gamma(\cO_K[1/p],\Lambda(\eta)(1))[1]\right).	
$$
In the same way
$$
	\Det_{\Lambda_\cO}(\Lambda_\cO(_\fra\zeta(\eta)))(\Div(\tau_\eta))=
	\Det_{\Lambda_\cO}(\cJ_\Lambda(\zeta(\eta)))
$$
and we have
$$
	\Div(\kappa_\eta\verk\tau_\eta)=\Div(\tau_\eta)+\Div(\kappa_\eta).
$$
By \ref{divisibility} the divisor $\Div(\kappa_\eta\verk \tau_\eta)$ is effective.
\begin{lemma}
The divisor $\Div(\tau_\eta)$ is effective and we have
$$
	(\N\fra-\sigma_\fra)^{-1}\Det_{\Lambda_\cO}(\Lambda_\cO(_\fra\zeta(\eta)))=
	\Det_{\Lambda_\cO}(\cJ_\Lambda(\zeta(\eta))).
$$
\end{lemma}
\begin{proof}
Consider the inclusions
$$
	\Lambda_\cO(_\fra\zeta(\eta))\subset \cJ_\Lambda(\zeta(\eta))\subset \Lambda_\cO(\zeta(\eta)).
$$
As the quotient of the last inclusion is pseudo-null, we get 
$$
	\Det_{\Lambda_\cO}(\cJ_\Lambda(\zeta(\eta)))=\Det_{\Lambda_\cO}(\Lambda_\cO(\zeta(\eta))).
$$
We have an
exact sequence
$$
	0\to \Lambda_\cO (_\fra\zeta(\eta))  \to \Lambda_\cO(\zeta(\eta))\to 
	\Lambda_\cO/(\N\fra-\sigma_\fra)\Lambda_\cO\to 0.
$$
It follows
$$
	\Det_{\Lambda_\cO}(\Lambda_\cO (\zeta(\eta)))=(\N\fra-\sigma_\fra)^{-1}\Det_{\Lambda_\cO}(\Lambda_\cO(_\fra\zeta(\eta))).
$$
\end{proof}
With this result, we see that $\Div(\kappa_\eta)$ is effective as well
and one has
$$
	\Det_{\Lambda_\cO}(\cJ_\Lambda)\subset\Det_{\Lambda_\cO}(\cJ_\Lambda)(\Div(\kappa_\eta)).
$$
The statement of the $\Lambda$-main-conjecture is that $\Div(\kappa_\eta)=0$.
To show this we consider all characters $\eta$ of level $n$ big enough at the
same time. Recall that $K_n(\frf_\eta)$ is the
compositum $K_nK(\frf_\eta)$ and that we defined $G_n(\frf_\eta):=\Gal(K_n(\frf_\eta)/K)$.
As the divisors $\Div(\kappa_\eta)$ are effective, they vanish precisely when
$$
	\sum_{\eta}\Div(\kappa_\eta)=0,
$$
where the sum is over all $\eta\in \wh{G}_n(\frf_\eta) \smallsetminus \wh{G}_{n-1}(\frf_\eta)$.
From the above we have 
$$
	\bigotimes_\eta\Det_{\Lambda_\cO}(\cJ_\Lambda(\zeta(\eta)))\subset
	\bigotimes_\eta\Det_{\Lambda_\cO}\left( R\Gamma(\cO_K[1/p],\Lambda(\eta)(1))[1]\right),	
$$
where the tensor products are again taken over 
$\eta\in \wh{G}_n(\frf_\eta) \smallsetminus \wh{G}_{n-1}(\frf_\eta)$.
By the above lemma this can be formulated as follows:
\begin{equation}\label{detereq}
	\bigotimes_\eta(\N\fra-\sigma_\fra)^{-1}\Det_{\Lambda_\cO}(\Lambda_\cO({_\fra\zeta(\eta)}))\subset
	\bigotimes_\eta\Det_{\Lambda_\cO}\left( R\Gamma(\cO_K[1/p],\Lambda(\eta)(1))[1]\right).
\end{equation}	
To show that this is an equality of line bundles on $\Spec\Lambda_\cO$ we use Nakayama's lemma.
Consider the augmentation map
$$
	\iota:\Lambda_\cO\to \cO_p.
$$
We denote also by $\iota$ the induced map $\iota:\Spec\cO_p\to \Spec\Lambda_\cO$. We have to show that after applying $L\iota^*$ to both sides in (\ref{detereq})
we get equality.
\begin{lemma}\label{reductionlemma}
	Let $\iota$ be as above, then $L\iota^*(\kappa_\eta\verk\iota_\eta)$ is the 
	map induced by the inclusion $_\fra\zeta_{K}(\eta)\in H^1(\cO_{K}[1/p],\cO_p(\eta)(1))$
	$$
		L\iota^*(\kappa_\eta\verk\iota_\eta):\cO_p({_\fra\zeta_{K}(\eta)})\to
		R\Gamma(\cO_{K}[1/p],\cO_p(\eta)(1))[1].
	$$
\end{lemma}
\begin{proof}
The map $\iota:\Lambda_\cO\to \cO_p$ induces a map of $\Gal(\bar{K}/K)$-modules
$$
	\Lambda(\eta)\to \cO_p(\eta)
$$
and hence an isomorphism 
$$
	L\iota^*R\Gamma (\cO_K[1/p],\Lambda(\eta)(1))\isom R\Gamma (\cO_K[1/p],\cO_p(\eta)(1)).
$$
Using the definition of $_\fra\zeta(\eta)$, we see that $L\iota^*(\kappa_\eta\verk\iota_\eta)$ is the 
map induced by the inclusion $_\fra\zeta_{K}(\eta)\in H^1(\cO_{K}[1/p],\cO_p(\eta)(1))$
$$
	L\iota^*(\kappa_\eta\verk\iota_\eta):\cO_p{_\fra\zeta_{K}(\eta)}\to
	R\Gamma(\cO_{K}[1/p],\cO_p(\eta)(1))[1].
$$
\end{proof}
 As the map $\Lambda(\eta)\to \cO_p(\eta)$ maps $\gamma\otimes t_p(\eta)\mapsto \eta(\gamma)t_p(\eta)$
we see that
$$
	\left((\N\fra-\sigma_\fra)^{-1}\Det_{\Lambda_\cO}(\Lambda_\cO({_\fra\zeta(\eta)}))\right)\otimes_{\Lambda_\cO}\cO_p=(\N\fra-\eta^{-1}(\fra))^{-1}\Det_{\cO_p}\cO_p({_\fra\zeta_K(\eta)})
$$
we get after applying $L\iota^*$ to both sides in (\ref{detereq}):
$$
\bigotimes_\eta(\N\fra-\eta^{-1}(\fra))^{-1}\Det_{\cO_p}\cO_p({_\fra\zeta_K(\eta)})\subset 
\bigotimes_\eta R\Gamma (\cO_K[1/p],\cO_p(\eta)(1)),
$$
where $\eta\in \wh{G}_n(\frf_\eta) \smallsetminus \wh{G}_{n-1}(\frf_\eta)$.
Application of corollary \ref{generationbyEuler} to $L=K_n(\frf_\eta)$ and $F=K_{n-1}(\frf_\eta)$
gives that this is in fact an equality. This ends the proof of the $\Lambda$-main-conjecture.

\section{Proof of the $\Omega$-main-conjecture}\label{proof-of-omega-mc}
The proof of the $\Omega$-main-conjecture essentially reduces, using an observation of Burns and Greither,
to the $\Lambda$-main-conjecture plus conjecture \ref{mu=0} which, as
we stress again, is a theorem in the case where $p$ is split in $K$ and $p$ does not divide
$6$.

\subsection{Preliminary reductions}

Recall that $\Omega=\Lambda(\cG_\frf)\isom \bbZ_p[\Delta][[S,T]]$.
\begin{lemma}
Let $\cO_p$ contain the values of all characters of $\Delta$, then it suffices
to prove \ref{omega-mc} for $\Omega_\cO$.
\end{lemma}
\begin{proof}
Inside $\Det_\Omega H^1(\cO_K[1/p\frf],\Omega(1))\otimes Q(\Omega)$ we have
two $\Omega$-modules $\Det_\Omega \cJ_\Omega(\zeta(\frf))$ and $\Det_\Omega R\Gamma(\cO_K[1/p\frf],\Omega_\cO(1))[1]$.
To check that they are equal, we can make the faithfully flat base extension
$\Omega\to \Omega_\cO$.
\end{proof}

We assume now that $\cO_p$ contains the values of the characters of $\Delta$.
We need the following result about the ring 
$\Omega_\cO\isom \cO_p[\Delta][[S,T]]$. 
\begin{lemma}\label{normalizationlemma}
The normalization $\wt\Omega_\cO$ of $\Omega_\cO$ inside of $Q(\Omega_\cO)$ is given by
$$
	\wt\Omega_\cO\isom \prod_{\chi\in\wh\Delta}\Lambda(\chi).
$$
In particular, $\Omega_\cO\otimes_{\cO_p}E_p\isom \wt\Omega_\cO\otimes_{\cO_p}E_p$.
\end{lemma}
\begin{proof}
This follows from the fact that $\Omega_\cO\subset \prod_{\chi\in\wh\Delta}\Lambda(\chi)$
and that the latter ring is normal. 
\end{proof}
We can now prove the first part of the equivariant main conjecture \ref{omega-mc}
\begin{corollary}
The module $H^2(\cO_K[1/p\frf],\Omega_\cO(1))$ is an $\Omega_\cO$-torsion module and
$H^1(\cO_K[1/p\frf],\Omega_\cO(1))$ has $\Omega_\cO$-rank one.
\end{corollary}
\begin{proof}
It follows from \cite{Fu-Ka} 1.6.5 (3) that 
\begin{equation}\label{coeff}
	R\Gamma(\cO_K[1/p\frf],\Omega_\cO(1))\otimes_{\Omega_\cO}^\bbL\wt\Omega_\cO\isom
	R\Gamma(\cO_K[1/p\frf],\wt\Omega_\cO(1))
\end{equation}
and as $H^0(\cO_K[1/p\frf],\Omega_\cO(1))=0$ (by (\ref{h0vanishing})) this implies that 
$$
	H^2(\cO_K[1/p\frf],\Omega_\cO(1))\otimes_{\Omega_\cO}\wt\Omega_\cO\isom  H^2(\cO_K[1/p\frf],\wt\Omega_\cO(1))
	\isom \prod_{\chi\in\wh\Delta}H^2(\cO_K[1/p\frf],\Lambda(\chi)(1)).
$$
We have an exact sequence
$$
	H^2(\cO_K[1/p],\Lambda(\chi)(1))\to H^2(\cO_K[1/p\frf],\Lambda(\chi)(1))\to 
	\bigoplus_{\frl\mid\frf,\frl\nmid p}H^1(\kappa(\frl),\Lambda(\chi)^{I_\frl})
$$
and by \ref{lambda-mc} the first module is a torsion $\Lambda_\cO$-module. The last one
is  the cokernel of $\Lambda_\cO\xrightarrow{1-\chi(\frl)\Frob^{-1}_\frl}\Lambda_\cO$, which
is also torsion. This shows that $H^2(\cO_K[1/p\frf],\wt\Omega_\cO(1))$ is torsion.
As $Q(\Omega_\cO)=Q(\wt\Omega_\cO)$, it follows that
$$
	H^2(\cO_K[1/p\frf],\Omega_\cO(1))\otimes_{\Omega_\cO}Q(\Omega_\cO)=0,
$$
which proves
that $H^2(\cO_K[1/p\frf],\Omega_\cO(1))$ is $\Omega_\cO$-torsion. Moreover, one gets from
(\ref{coeff}) an exact sequence
\begin{multline*}
	\Tor_2^{\Omega_\cO}(H^2(\cO_K[1/p\frf],\Omega_\cO(1)),\wt\Omega_\cO)\to 
	H^1(\cO_K[1/p\frf],\Omega_\cO(1))\otimes_{\Omega_\cO}\wt\Omega_\cO\to \\ \to H^1(\cO_K[1/p\frf],\wt\Omega_\cO(1)).
\end{multline*}
We have 
$$
	H^1(\cO_K[1/p],\wt\Omega_\cO(1))\to H^1(\cO_K[1/p\frf],\wt\Omega_\cO(1))\to
	\bigoplus_{\frl\mid\frf,\frl\nmid p}H^0(\kappa(\frl),\wt\Omega_\cO^{I_\frl})
$$
and the last group is zero as the Frobenius acts non-trivially on $\wt\Omega_\cO$.
It follows from \ref{normalizationlemma} and \ref{lambda-mc} and the fact that 
$H^2(\cO_K[1/p\frf],\Omega_\cO(1))$ is a torsion module
that $H^1(\cO_K[1/p\frf],\Omega_\cO(1))$ has 
$\Omega_\cO$-rank one.
\end{proof}
To prove the rest of the equivariant main conjecture \ref{omega-mc} we want to
use the following lemma taken from Flach \cite{Flach} (recall that $\Lambda_\cO$ and $\Omega_\cO$ are
products of local rings):
\begin{lemma}[\cite{Flach} 5.3]\label{localizationlemma} Let $R=\Lambda_\cO$ or $R=\Omega_\cO$ and $Q(R)$ be the
total quotient ring. Let
$M$ and $N$ be two invertible $R$-submodules of some invertible $Q(R)$-module $D$, 
then $M=N$ if and only if for all height $1$ prime ideals $\frq$ of $R$
  one has $M_\frq=N_\frq$ inside $D_\frq$.
\end{lemma}
Inside $ Q(\Omega_\cO)$ we have
two rank one $\Omega_\cO$-modules:
$$
	\Det_{\Omega_\cO}\left(H^1(\cO_K[1/p\frf],\Omega_\cO(1))/\cJ_\Omega(\zeta(\frf))\right)\mbox{ and } 
	\Det_{\Omega_\cO}(H^2(\cO_K[1/p\frf],\Omega_\cO(1))).
$$
To show that these are equal we can by \ref{localizationlemma} localize at all
height one primes of $\Omega_\cO$. We distinguish two cases following Burns and Greither:
\begin{definition}
A prime ideal  $\frq\subset \Omega_\cO$ of height one is called \emph{regular}
if $p\notin \frq$. If $p\in \frq$, the prime ideal is called \emph{singular}.
\end{definition}
The proof of the $\Omega$-main-conjecture in these two cases is given in the
next two sections.

\subsection{Proof for regular prime ideals}

First note the following consequence of lemma \ref{normalizationlemma}:
\begin{lemma}\label{regularlocalization}
Let $\frq\subset \Omega_\cO$ be a regular prime ideal of height one, then
$$
	(\Omega_\cO)_\frq\isom (\wt\Omega_\cO)_\frq\isom \prod_{\chi\in\wh\Delta}\Lambda(\chi)_\frq.
$$
\end{lemma}
\begin{proof}
As $p$ is invertible in $(\Omega_\cO)_\frq$ both rings are localizations of 
$\Omega_\cO\otimes_{\cO_p}E_p$ resp. $\wt\Omega_\cO\otimes_{\cO_p}E_p $, which
agree by lemma \ref{normalizationlemma}.
\end{proof}
It follows that for regular $\frq $ (using again \cite{Fu-Ka} 1.6.5. (3)) 
$$
R\Gamma(\cO_K[1/p\frf],\Omega_\cO(1))_\frq\isom R\Gamma(\cO_K[1/p\frf],\wt\Omega_\cO(1))_\frq\isom
	\prod_{\chi\in\wh\Delta} R\Gamma(\cO_K[1/p\frf],\Lambda(\chi)(1))_\frq.
$$
\begin{lemma}\label{decompositionlemma}
The image of $\cJ_\Omega(\zeta(\frf))_\frq\subset H^1(\cO_K[1/p\frf],\Omega_\cO(1))_\frq$ under
the isomorphism
$$
	H^1(\cO_K[1/p\frf],\Omega_\cO(1))_\frq\simeq \prod_{\chi\in\wh\Delta}H^1(\cO_K[1/p\frf],\Lambda(\chi)(1))_\frq
$$
is given by
$$
	\prod_{\chi\in\wh\Delta}\cJ_\Lambda(\zeta(\chi,\frf))_\frq
$$
with $\zeta(\chi,\frf)$ defined in corollary \ref{localizationcor}.
\end{lemma}
\begin{proof}
First note that $(\cJ_\Omega)_\frq=(\Omega_\cO)_\frq=\prod_{\chi\in\wh\Delta}\Lambda(\chi)_\frq$
and that the map $(\Omega_\cO)_\frq\to \Lambda(\chi)_\frq$ is given by the projector 
$p_{\chi^{-1}}$. Recall that ${_\fra\zeta(\chi)}=\prolim_n {_\fra\zeta_{K_n}(\chi)}$
and that $_\fra\zeta(\frf)=\prolim_n {_\fra\zeta_{\frf p^n}}$. Let $\frf_\chi$ be the 
conductor of $\chi$. By definition
$_\fra\zeta_{K_n}(\chi)=\tr_{K(\frf_\chi p^m)/K_n} {_\fra}\zeta_{\frf_\chi p^m}$,
where $m$ is the smallest integer such that $K_n\subset K(\frf_\chi p^m)$. 
Let $t$ be the smallest integer such that $K(\frf_\chi p^m)\subset K(\frf p^t)$.
The distribution relation gives
$$
	\tr_{K(\frf p^t)/K_n}(_\fra\zeta_{\frf p^t}\otimes t_p(\chi))=\prod_{\frl\mid\frf p^t,\frl\nmid p^n}(1-\chi(\frl)\Frob_\frl^{-1}){_\fra\zeta_{K_n}(\chi)}.
$$
Consider $\chi$ as a character on $G:=\ker(G(\frf p^t)\to G_n)$.
By lemma \ref{zetaidentification} with
$G=G$ and $H=G_n$ we get that 
$\tr_{K(\frf p^t)/K_n}(_\fra\zeta_{\frf p^t}\otimes t_p(\chi))=p_{\chi^{-1}}(_\fra\zeta_{\frf p^t})$.
This implies that the map
$$
	H^1(\cO_K[1/p\frf],\Omega_\cO(1))_\frq\to H^1(\cO_K[1/p\frf],\Lambda(\chi)(1))_\frq
$$
maps $_\fra\zeta(\frf)$ to $\prod_{\frl\mid\frf,\frl\nmid p}(1-\chi(\frl)\Frob_\frl^{-1}){_\fra\zeta(\chi)}$.
The last element is 
$\zeta(\chi,\frf)$ defined in corollary \ref{localizationcor}.
\end{proof}
Consider the map
$$
	(\kappa_\frf)_\frq:\cJ_\Omega(\zeta(\frf))_\frq\to R\Gamma(\cO_K[1/p\frf],\Omega_\cO(1))_\frq[1].
$$
Using the above lemma this coincides with
$$
	\prod_{\chi\in\wh\Delta}(\wt\kappa_\chi)_\frq:\prod_{\chi\in\wh\Delta}\cJ_\Lambda(\zeta(\chi,\frf))_\frq\to 
	\prod_{\chi\in\wh\Delta} R\Gamma(\cO_K[1/p\frf],\Lambda(\chi)(1))_\frq.
$$
By corollary \ref{localizationcor} each $\wt\kappa_\chi$ induces an isomorphism
$$
\Det_{(\Lambda_\cO)_\frq} \left( H^1(\cO_{K}[1/p\frf],\Lambda(\chi)(1))_\frq/\cJ_\Lambda(\zeta(\chi,\frf))_\frq\right)\simeq \Det_{(\Lambda_\cO)_\frq} H^2(\cO_{K}[1/p\frf],\Lambda(\chi)(1))_\frq.
$$
This proves the $\Omega$-main-conjecture for regular prime ideals.

\subsection{Proof for singular prime ideals}\label{proofforsingular}

Let $\frq\subset \Omega_\cO$ be a singular prime ideal (i.e., $p\in \frq$).
Then by our assumption $H^2(\cO_K[1/p\frf],\Omega_\cO(1))_\frq=0$ and
we get
$$
	\Det_{(\Omega_\cO)_\frq}R\Gamma(\cO_K[1/p\frf],\Omega_\cO(1))_\frq[1] =
	\Det_{(\Omega_\cO)_\frq}H^1(\cO_K[1/p\frf],\Omega_\cO(1))_\frq .
$$
Consider 
$$
	(\kappa_\frf)_\frq:\cJ_\Omega(\zeta(\frf))_\frq\to R\Gamma(\cO_K[1/p\frf],\Omega_\cO(1))_\frq .
$$
This gives, as by the above the divisor $\Div(\kappa_\frf)_\frq$ is effective, an inclusion
\begin{equation}\label{detinclusion}
	\Det_{(\Omega_\cO)_\frq}\cJ_\Omega(\zeta(\frf))_\frq\subset \Det_{(\Omega_\cO)_\frq}H^1(\cO_K[1/p\frf],\Omega_\cO(1))_\frq .
\end{equation}
We now use an idea of Witte.
Choosing generators for both modules we see that there is an element
$u\in (\Omega_\cO)_\frq$ such that
$$
	\Det_{(\Omega_\cO)_\frq}\cJ_\Omega(\zeta(\frf))_\frq=u
	\Det_{(\Omega_\cO)_\frq}H^1(\cO_K[1/p\frf],\Omega_\cO(1))_\frq.
$$
We want to show that $u$ is a unit in $(\Omega_\cO)_\frq$.
Consider the normal ring homomorphism $(\Omega_\cO)_\frq\to (\wt\Omega_\cO)_\frq$.
An element $u\in (\Omega_\cO)_\frq$ is a unit if and only if it is a unit
in $(\wt\Omega_\cO)_\frq$. Thus it suffices to show that after extending
scalars in (\ref{detinclusion}) we get an equality. By lemma \ref{decompositionlemma}
the map $(\kappa_\frf)_\frq$  decomposes after this base extension into 
$$
\prod_{\chi\in\wh\Delta}(\wt\kappa_\chi)_\frq:\prod_{\chi\in\wh\Delta}\cJ_\Lambda(\zeta(\chi,\frf))_\frq\to 
	\prod_{\chi\in\wh\Delta} R\Gamma(\cO_K[1/p\frf],\Lambda(\chi)(1))_\frq.
$$
By corollary \ref{localizationcor} each $\wt\kappa_\chi$ induces an isomorphism
$$
\Det_{(\Lambda_\cO)_\frq} \left( H^1(\cO_{K}[1/p\frf],\Lambda(\chi)(1))_\frq/\cJ_\Lambda(\zeta(\chi,\frf))_\frq\right)\simeq \Det_{(\Lambda_\cO)_\frq} H^2(\cO_{K}[1/p\frf],\Lambda(\chi)(1))_\frq.
$$
This shows that the element $u\in (\Omega_\cO)_\frq$ becomes a unit in
$(\wt\Omega_\cO)_\frq$. Thus $u$ is already a unit in $ (\Omega_\cO)_\frq$ and we get
equality in (\ref{detinclusion}), which proves the $\Omega$-main-conjecture for singular prime ideals.


\begin{thebibliography}{MMMM}
\bibitem[Bl]{Bley}W. Bley: Equivariant Tamagawa number conjecture for abelian extensions of a quadratic imaginary field.  Doc. Math.  11  (2006), 73--118.
\bibitem[BF]{BF1}D. Burns and M. Flach: Tamagawa numbers for motives with (non-commutative) coefficients.  Documenta Mathematica 6 (2001), 501-569.
\bibitem[BG]{Bu-Gr}D. Burns and C. Greither: On the equivariant Tamagawa Number Conjecture for Tate Motives, Invent. Math., 153, pp 303-359, 2003.
\bibitem[dS]{deShalit}E. de Shalit:  Iwasawa theory of elliptic curves with complex multiplication, Academic Press, 1987.
\bibitem[Fo]{Fontaine}J.-M. Fontaine: J.-M. Fontaine: Valeurs sp\'eciales des fonctions $L$ des motifs, S\'eminaire Bourbaki, Vol. 1991/92. 
Ast\'erisque No. 206, (1992), Exp. No. 751, 4, 205--249. 
\bibitem[Fl]{Flach}M. Flach:  The equivariant Tamagawa number conjecture: A survey, in {\em Stark's Conjectures: Recent Work and New Directions.}  Contemporary Mathematics, vol. 358, AMS, 2004.
\bibitem[Fu-Ka]{Fu-Ka}T. Fukaya, K. Kato: A formulation of conjectures on $p$-adic zeta functions in noncommutative Iwasawa theory.  Proceedings of the St. Petersburg Mathematical Society. Vol. XII,  1--85, Amer. Math. Soc. Transl. Ser. 2, 219, Amer. Math. Soc., Providence, RI, 2006.
\bibitem[Gi]{Gillard}R. Gillard:  Fonctions $L$ $p$-adiques des corps quadratiques imaginaires et de leurs extensions ab\'eliennes.  J. Reine Angew. Math.  358  (1985), 76--91. 
\bibitem[HoKi]{Ho-Ki}J. Hornbostel, G. Kings: On non-commutative twisting in \'etale and motivic cohomology.  Ann. Inst. Fourier (Grenoble)  56  (2006),  no. 4, 1257--1279. 
\bibitem[HK]{HK1}{A. Huber, G. Kings:} Bloch-Kato conjecture and Main Conjecture of Iwasawa theory for Dirichlet characters.  Duke Math. J.  119  (2003),  no. 3, 393--464.
\bibitem[It]{It}{K. Itakura}: Tamagawa number conjecture
and Iwasawa main conjecture for Dirichlet motives at the prime 2, preprint.
\bibitem[Ja]{Ja}{U. Jannsen}: On the $l$-adic cohomology of varieties
over number fields and its Galois cohomology, in: Ihara et al. (eds.): Galois
groups over $\bbQ$, MSRI Publication (1989)
\bibitem[JL]{JLthesis}{J. Johnson-Leung}: ETNC for imaginary quadratic fields, preprint.
\bibitem[Ka]{Kato} K. Kato: $p$-adic Hodge theory and values of zeta functions of modular forms, Asterisque, (2004), no. 295, ix, 117--290.
\bibitem[Ki]{Kings} G. Kings: The Tamagawa number conjecture for CM elliptic curves.  
	Invent. Math. {143}  (2001),  no. 3, 571--627.
\bibitem[KM]{knumum} F. Knudsen, D. Mumford: The projectivity of the moduli space of stable curves I: 
Preliminaries on `det' and `div', Math. Scand. (1976), 39:19-- 55.
\bibitem[La]{Lang}S. Lang: Elliptic Functions, Springer 1987.
\bibitem[Mi]{Milne} J. S. Milne: \'Etale cohomology, Princeton University Press, 1980.
\bibitem[NSW]{NSW} J. Neukirch, A. Schmidt, K. Wingberg: Cohomology of number fields.  Springer-Verlag 2000. 
\bibitem[Ru1]{Rubin}K. Rubin: The "main conjectures" of Iwasawa theory for imaginary 
quadratic fields, Invent. Math. 103 (1991), 25-68.
\bibitem[Ru2]{RubinMoreMC}K. Rubin: More "main conjectures" for imaginary quadratic fields.
Elliptic curves and related topics, 23-28, CRM Proc. Lecture Notes, 4, Amer. Math. Soc., Providence, RI, 1994
\bibitem[Ru3]{Rubin2}K. Rubin: Euler systems. Annals of Mathematics Studies, 147, Princeton University Press, Princeton, NJ, 2000 
\bibitem[Sch]{Schneider}P. Schneider: \"Uber gewisse Galoiscohomologiegruppen, Math. Z. 168
(1979), no. 2, 181--205. 
\bibitem[Ts]{Ts}T. Tsuji: Explicit reciprocity law and formal moduli for Lubin-Tate formal groups, J. reine angew. Math.,{\bfseries 569}
(2004), 103-173.
\bibitem[Wi]{Witte}M. Witte: On the equivariant main conjecture of Iwasawa theory, Acta Arith. 122 (2006), no. 3, 275--296.

\end{thebibliography}
\end{document}